\documentclass[pdflatex,sn-mathphys-num]{sn-jnl}

\usepackage{amsmath,amssymb,amsfonts,amsthm}
\usepackage{mathtools}
\usepackage{microtype}
\usepackage{enumitem}
\usepackage{cleveref}
\usepackage{braket}
\usepackage[dvipsnames]{xcolor}

\mathtoolsset{showonlyrefs}
\allowdisplaybreaks

\newtheorem{theorem}{Theorem}[section]
\newtheorem{proposition}[theorem]{Proposition}
\newtheorem{lemma}[theorem]{Lemma}
\newtheorem{corollary}[theorem]{Corollary}
\theoremstyle{remark}

\DeclareMathOperator{\conv}{conv}
\DeclareMathOperator{\cone}{cone}
\DeclareMathOperator{\clconv}{\overline{conv}}

\newcommand{\R}{\mathbb R}
\newcommand{\Smat}{\mathbb S}
\newcommand{\vx}{\mathbf{x}}
\newcommand{\vr}{\mathbf{r}}
\newcommand{\vc}{\mathbf{c}}
\newcommand{\cC}{\bar{\mathcal C}}
\DeclareRobustCommand{\Fset}{F}
\newcommand{\Gone}{G_1}
\newcommand{\Gtwo}{G_2}
\newcommand{\Ncone}{\mathcal N}
\newcommand{\Sigmacone}{\Sigma}

\newcommand{\Poct}{\mathcal P}

\def\mxX{\begin{pmatrix}
		1 & \vx^\top\\ \vx & X
	\end{pmatrix}}

\title[Nonnegative quadratics on a quadrant]{Nonnegative Quadratics over a Quadrant with a Bilinear Constraint}
\author[1]{\fnm{Yipeng} \sur{Zhang}}
\author[1]{\fnm{Yuyuan} \sur{Ouyang}}
\author[1]{\fnm{Boshi} \sur{Yang}}
\affil[1]{\orgdiv{School of Mathematical and Statistical Sciences}, \orgname{Clemson University}, \orgaddress{\city{Clemson}, \postcode{29634}, \state{South Carolina}, \country{USA}}}

\begin{document}

\abstract{%
We study quadratic polynomials that are nonnegative on the non-compact set
\[
\Fset:=\{(x_1,x_2)\in\mathbb R^2:\ x_1\ge 0,\ x_2\ge 0,\ x_1x_2\le 1\}.
\]
All extreme rays of the cone of nonnegative quadratic polynomials are characterized on this set. The characterization allows us to study parameterized valid inequalities for quadratic convexifications
involving $\Fset$, which yields a
semidefinite representation of its lifted convex hull in the
quadratic space.
Our analysis on extreme ray characterization separates the positive-semidefinite (PSD) and non-PSD branches,
reduces the latter to boundary nonnegativity, and classifies the boundary
contacts of the relevant extreme rays.  
By reparameterization, we also find a non-trivial and non-permutation-symmetric six-dimensional linear section of the cone of nonnegative homogeneous ternary octics that are sum-of-squares. We also show that our lifted convex hull result yields a degree-bounded preordering certificate of nonnegative quadratics on $\Fset$ and a degree-bounded certificate for a family of nonnegative quartics on the half-strip.
}
\maketitle

\section{Introduction} \label{sec:intro}

A fundamental problem in quadratically constrained quadratic programming (QCQP) is to convexify the feasible region in the lifted space. For a set
\(F\subseteq\R^n\), its closed lifted convex hull is defined by
\[
    \cC(F)
    :=
    \clconv\left\{
        \begin{pmatrix}1\\ \vx\end{pmatrix}
        \begin{pmatrix}1\\ \vx\end{pmatrix}^{\!\top}
        \;\middle|\;
        \vx\in F
    \right\}.
\]
Writing a member of \(\cC(F)\) as
$
    Y=
    \begin{pmatrix}
        1 & \vx^\top\\
        \vx & X
    \end{pmatrix},
$
the quadratic objective function
\(\vx^\top Q\vx+2\vc^\top\vx+\gamma\) can be linearized as
\(Q\mathbin{\bullet}X+2\vc^\top\vx+\gamma\), where $Q \bullet X$ is the Frobenius product of symmetric matrices $Q$ and $X$. Minimizing the latter over
\(\cC(F)\) gives the same infimum as minimizing the original
quadratic over \(F\). Exact descriptions of \(\cC(F)\) are thus universal
convex reformulations for quadratic objectives on \(F\), while partial
descriptions provide valid inequalities that strengthen standard
semidefinite relaxations.

Characterizing \(\cC(F)\) is computationally intractable in general, but semidefinite representations are known for several important special structures. Low-dimensional simplices and boxes admit representations based on positive semidefiniteness together with entrywise nonnegativity or reformulation--linearization technique (RLT) constraints \cite{AnstreicherBurer2010, sherali2013reformulation}. The trust-region subproblem and its variants furnish another influential class for which semidefinite and
second-order-cone constraints yield exact formulations or strong relaxations
\cite{anstreicher2017kronecker, BurerAnstreicher2013, burer2025slightly, kelly2023semidefinite, kilincc2025strength, rendl1997semidefinite, Sturm.Zhang.2003, YangBurer2016}. These examples illustrate a
broader theme: special geometric structure can sometimes turn a
semi-infinite collection of valid quadratic inequalities into a small conic
system. They also show that exactness is sensitive to how the defining
constraints interact.

In this paper, we study the set
\[
    \Fset
    :=
    \{\vx\in\R^2\mid x_1\geq0,\ x_2\geq0,\ x_1x_2\leq1\}.
\]
The set \(\Fset\) is a basic non-compact model of an upper bound on the product
of two nonnegative quantities. Product bounds arise explicitly in bilinear
models and can also be generated by range-reduction procedures in global
optimization \cite{AnstreicherBurerPark2021}. They also occur as local
structures in bipartite bilinear programs \cite{chen2022tightening}.
Unlike the classical bounded McCormick setting, \(\Fset\) retains the
nonnegativity lower bounds but imposes no finite upper bound on either
variable. It therefore isolates the lifted geometry created by the
interaction between nonnegativity and a product upper bound, without the
additional RLT consequences generated by finite individual upper bounds.

Despite its elementary two-dimensional description, the exact lifted convex hull of $\Fset$ is not covered by any existing general results to our knowledge. Joyce and Yang \cite{JoyceYang2024} showed that, under a non-intersecting assumption, adding a quadratic inequality to a set with known lifted convex hull amounts to intersecting that hull with the linearization of the added inequality. A natural decomposition of $\Fset$ is
$
    \Fset
    =
    \R_+^2
    \cap
    \{\vx\in\R^2\mid x_1x_2\leq1\}.
$
However, the boundary \(x_1x_2=1\) has a branch in the third quadrant, which is not contained in \(\mathbb{R}_{+}^{2}\), so the required non-intersecting condition fails. Indeed, Joyce and Yang
\cite{JoyceYang2024} explicitly identify \(\Fset\) as a simple but interesting example beyond the reach of their result. 

Complementary evidence of this difficulty comes from the convexification of bounded bilinear products. Belotti et al. \cite{belotti2010valid} show that, even in the graph space \((x_1,x_2,z)\) with \(z=x_1x_2\), imposing a nontrivial bound on \(z\) generates an infinite family of linear valid inequalities beyond the standard McCormick or RLT inequalities. Recently, Anstreicher et al. \cite{AnstreicherBurerPark2021} show that the infinite family can be represented equivalently by a second-order-cone constraint. They also identify the extension of known convex-hull descriptions for the complete five-variable quadratic system $(x_1,x_2,x_1^2,x_1x_2,x_2^2)$ to incorporate bounds on \(x_1x_2\) as an interesting direction for further research. The set \(\cC(\Fset)\) addresses a complementary unbounded regime of precisely this interaction: it retains the nonnegativity lower bounds and an upper bound on the product, but imposes no finite upper bounds on the individual variables. Consequently, its description requires not only convexifying the bilinear graph but also
characterizing how the product moment \(X_{12}\) couples with \(x_1,x_2,X_{11}\), and \(X_{22}\).

The significance of \(\cC(\Fset)\) is not limited to this two-variable
model. Let \(G\subseteq\R^n\) be the feasible region of a larger QCQP, and
suppose that, for some \(i\neq j\) and \(\tau>0\),
\[
    x_i\geq0,
    \qquad
    x_j\geq0,
    \qquad
    x_ix_j\leq\tau
    \qquad
    \text{for every }\vx\in G.
\]
For a lifted matrix \(Y\), let \(Y[\{0,i,j\}]\) denote the principal
submatrix indexed by the homogenizing coordinate and coordinates \(i\) and
\(j\). Then every \(Y\in\cC(G)\) satisfies
\[
    D_\tau Y[\{0,i,j\}]D_\tau\in\cC(\Fset),
    \quad \text{where} \quad
    D_\tau
    :=
    \mathrm{diag}\bigl(1,\tau^{-1/2},\tau^{-1/2}\bigr).
\]
This inclusion holds for every rank-one feasible lift and is preserved under
convex combinations and limits. Consequently, every affine inequality valid
for \(\cC(\Fset)\) pulls back to a valid inequality involving only
\((x_i,x_j,X_{ii},X_{ij},X_{jj})\) in the lifted formulation of \(G\). These
local constraints can be imposed for every available product bound and
combined with box-dependent McCormick or RLT inequalities whenever finite
variable bounds are also available. Although such pairwise constraints need
not describe \(\cC(G)\) exactly when additional variables create further
coupling, they provide reusable strengthenings that capture the complete
lifted geometry of each selected product-bound substructure.

Our analysis uses the dual viewpoint of nonnegative quadratic functions. For
\(R\in\Smat^2\), \(\vr\in\R^2\), and \(p\in\R\), let
\begin{equation}\label{eq:q-def-intro}
    q_{R,\vr,p}(\vx)
    :=
    \vx^\top R\vx+2\vr^\top\vx+p,
\end{equation}
and define
\[
    \Ncone(\Fset)
    :=
    \Set{\begin{pmatrix}
        p & \vr^\top\\\vr & R
    \end{pmatrix}|q_{R,\vr,p}(\vx)\geq0
    \quad\forall\,\vx\in\Fset}.
\]
In this paper, we may either us notation $\begin{pmatrix}
        p & \vr^\top\\\vr & R
    \end{pmatrix}\in\Ncone(\Fset)$ to emphasize the coefficients of a quadratic nonnegative on $F$, or $q_{R,\vr,p}\in \Ncone(\Fset)$ to emphasize the a nonnegative quadratic of interest.
Every member of \(\Ncone(\Fset)\) gives the valid lifted inequality
$
    R\mathbin{\bullet}X+2\vr^\top\vx+p\geq0
$
for any $\vx$ and $X$ described by $\cC(\Fset)$, 
and every affine inequality valid for \(\cC(\Fset)\) arises in this way. To
state the duality precisely, define the homogenized moment cone
\begin{align}
    \label{eq:KF}
    \mathcal K(\Fset)
    :=
    \overline{\cone}\left\{
        \begin{pmatrix}1\\ \vx\end{pmatrix}
        \begin{pmatrix}1\\ \vx\end{pmatrix}^{\!\top}
        \;\middle|\;
        \vx\in\Fset
    \right\}.
\end{align}
One has
\begin{align}
\label{eq:KNC_relations}    
    \mathcal K(\Fset)^*=\Ncone(\Fset),
    \qquad
    \mathcal K(\Fset)=\Ncone(\Fset)^*,
    \qquad
    \cC(\Fset)
    =
    \mathcal K(\Fset)\cap\{Y\mid Y_{00}=1\}.
\end{align}
Thus the extreme rays of \(\Ncone(\Fset)\) identify the irreducible sources
of valid inequalities for \(\cC(\Fset)\).

The following extreme-ray classification is the common structural foundation
for the two principal consequences of the paper: the exact lifted
convexification and the polynomial sum-of-squares result. Its proof is deferred to Section~\ref{sec:NF}.

\begin{theorem}\label{thm:main}
The extreme rays of \(\Ncone(\Fset)\) are precisely the rays spanned by the
following quadratics. Specifically, \(\Ncone(\Fset)\) is the conic hull of the following extreme rays:
    \begin{enumerate}
    \item The interior affine-square rays: \(\ell(\vx)^2\), where \(\ell\)
    is a nonconstant affine function whose zero set meets
    \(\operatorname{int}(\Fset)\).

    \item The axis-square rays: \(x_1^2\) and \(x_2^2\).

    \item The coordinate rays: \(x_1\), \(x_2\), and \(x_1x_2\).

    \item The hyperbola ray: \(1-x_1x_2\).

    \item The lifted tangent rays:
    \[
        h_w(\vx)
        :=
        x_1+w^2x_2-2wx_1x_2,
        \qquad w>0.
    \]

    \item The bitangent rays:
    \begin{align*}
        g^1_{u,v}(\vx)
        &:=(x_1-u)^2
        +(v-u)x_2\bigl(2v^2-(3v-u)x_1\bigr),
        &&0\leq u<v,\\
        g^2_{u,v}(\vx)
        &:=(x_2-u)^2
        +(v-u)x_1\bigl(2v^2-(3v-u)x_2\bigr),
        &&0\leq u<v.
    \end{align*}
\end{enumerate}
\end{theorem}

The extreme-ray classification is itself our first contribution. Its proof
separates the cases \(R\succeq0\) and \(R\not\succeq0\) and, in the latter
case, reduces extremality to a small collection of boundary-contact
configurations. Within the complete dual cone, the classification recovers,
as extreme rays, the lifted tangent inequalities known from \cite{belotti2010valid}. It also identifies two symmetric two-parameter families. We call them \emph{bitangent rays}, because each of them has second-order contact with a coordinate-axis boundary component and with the positive hyperbola. 

Our second contribution is a complete description of \(\cC(\Fset)\). Within the Shor--RLT system, the inequalities generated by the lifted tangent rays reduce to the known second-order-cone constraint $X_{12}^2\leq x_1x_2$ in \cite{AnstreicherBurerPark2021}.
The bitangent rays provide the missing inequalities that capture the coupling between the product moment \(X_{12}\) and one of the diagonal moments \(X_{11}\) and \(X_{22}\). Although the bitangent inequalities form two continuously parameterized families, they admit fixed-size lifted conic representations. Consequently,
\(\cC(\Fset)\) admits an explicit semidefinite extended
formulation with two scalar auxiliary variables. Moreover, apart from the
standard Shor positive-semidefinite block, all constraints in the formulation
are linear or second-order-cone representable.

We also have a few interesting new results from the algebraic point of view. A polynomial reparameterization
identifies \(\Ncone(\Fset)\) with \(\mathcal H\cap\Poct_{3,8}\), where
\(\mathcal H\) is a six-dimensional linear subspace of the space of ternary
octic forms and \(\Poct_{3,8}\) is the cone of nonnegative ternary octics.
The sum-of-squares (SOS) cone \(\Sigma_{3,8}\), consisting of sums of squares of ternary
quartic forms, is a tractable inner approximation of \(\Poct_{3,8}\).
Hilbert's theorem \cite{hilbert1888darstellung} gives the strict inclusion
$
    \Sigma_{3,8}\subsetneq\Poct_{3,8}.
$
Nevertheless, we prove that the two cones coincide on \(\mathcal H\):
\[
    \mathcal H\cap\Poct_{3,8}
    =
    \mathcal H\cap\Sigma_{3,8}.
\]
Different from a previous exceptional result \cite{Harris1999} in which $\mathcal H$ is the four-dimensional subspace of even symmetric homogeneous polynomials, in our result, nonnegativity and sums of squares coincide on a six-dimensional,
coordinate-wise-even but generally non-permutation-symmetric section of the
ternary-octic cone. In fact, although we have presented the problem from the viewpoint of lifted convexification, our research that yields this paper originated from this SOS analysis. In addition to the above nonnegative and SOS cone results, we also show that our characterization of $\Ncone(\Fset)$ yields a bounded-degree preordering certificate for nonnegative polynomials over $F$ and also a bounded-degree certificate for certain quartics on the half-strip \([0,1]\times[0,\infty)\).

This paper is organized as follows. Section~\ref{sec:CF} derives the exact conic representation of
\(\cC(\Fset)\) from Theorem~\ref{thm:main}.
Section~\ref{sec:poly-consequences} uses the same classification to prove the
nonnegative-octic/SOS identity and some bounded-degree certificates.
Section~\ref{sec:NF} proves the extreme-ray classification.
Section~\ref{sec:conclusion} concludes the paper.

\section{Main results: valid inequalities and a semidefinite representation of
\texorpdfstring{$\cC(\Fset)$}{C(F)}} \label{sec:CF}

In this section, we discuss the valid inequalities based on the list of extreme rays enumerated in Theorem \ref{thm:main}. We will link the enumerated extreme rays to previously known lifted convex hull results related to our set of interest $\Fset$, and show that the class of all bitangent rays is the only remaining piece needed to complete the picture: once all the bitangent ray cuts are added, the semidefinite representation of $\cC(\Fset)$ is readily available.

\subsection{Valid inequalities}
A straightforward relaxation of $\cC(\Fset)$ is Shor-RLT:
\begin{align}
\label{eq:CSR}
	\cC_{Shor-RLT}(\Fset) = \left\{
    \mxX 
    \succeq 0
	\ \middle| x_1, x_2\ge 0, 0\le X_{12}\le 1\right\}.
\end{align}
Here $x_1,x_2\ge 0$ and $X_{12}\le 1$ are valid inequalities induced by Shor relaxation, and $X_{12}\ge 0$ is that by RLT relaxation. 
It is straightforward to observe that other than lifted tangent and bitangent rays, all the extreme rays listed in Theorem \ref{thm:main} are associated with the valid inequalities induced Shor-RLT. Indeed, we have

\begin{theorem}\label{thm:exact-families}
The lifted convex hull is
\begin{equation}\label{eq:exact-family-description}
\begin{aligned}
\cC(\Fset)=&\biggl\{\mxX\in\cC_{Shor-RLT}(\Fset)| 
\\
&\qquad H_w(\vx, X), G_{u,v}^{1}(\vx, X), G_{u,v}^{2}(\vx, X) \ge 0,\ \forall w>0, 0\leq u<v\biggr\}
\end{aligned}
\end{equation}
where the linear lifting functions 
\begin{align}
    H_w(\vx, X):= & x_1+w^2x_2-2wX_{12}, 
    \\
    G^{1}_{u,v}(\vx, X) := & X_{11} - 2ux_1 + u^2 +2v^2(v-u)x_2 - (v-u)(3v-u)X_{12},\text{ and }
    \\
    G^{2}_{u,v}(\vx, X):= & X_{22}-2ux_2+u^2+2v^2(v-u)x_1
-(v-u)(3v-u)X_{12}
\end{align}
are associated with extreme lifted tangent and bitangent rays.
\end{theorem}

\begin{proof}
    This is an immediate consequence of the relations stated in \eqref{eq:KNC_relations}.
\end{proof}

In the sequel, we will say that the lifted tangent valid inequalities are linear inequalities of form $H_w(\vx)\ge 0$ and the bitangent valid inequalities are that of form $G^{1}_{u,v}(\vx, X)\ge 0$ or $G^{2}_{u,v}(\vx, X)\ge 0$. Here the lifted tangent valid inequalities have been studied in the literature; in fact, there has been a known result \cite{AnstreicherBurerPark2021} (see also \cite{belotti2010valid}) on characterizing a tighter cone $\cC_{H}(\Fset)$, which is the intersection of $\cC_{Shor-RLT}(\Fset)$ with all lifted tangent valid inequalities:
\begin{align}
    \label{eq:c1}
\cC_{H}(\Fset)=&\Set{\mxX\in\cC_{Shor-RLT}(\Fset)| H_w(\vx, X) \ge 0,\ \forall w>0}.
\end{align}
For completeness, we state the known result using our notation and include the proof below.

\begin{proposition}
\label{prop:family-I-soc}
We have
\begin{align}
    \cC_{H}(\Fset)=&\Set{\mxX\in\cC_{Shor-RLT}(\Fset)| \begin{pmatrix}x_1&X_{12}\\X_{12}&x_2\end{pmatrix}\succeq 0}.
\end{align}
\end{proposition}

\begin{proof}
If suffices to prove that
\begin{align}
\label{eq:family-I-det}
\begin{aligned}
    & \Set{(x_1,x_2,X_{12})| X_{12}\ge 0, x_1+w^2x_2-2wX_{12}\ge 0,\ \forall w>0} 
    \\
    = & \Set{(x_1,x_2,X_{12})| x_1,x_2,X_{12}\ge 0, X_{12}^2\leq x_1x_2}.
\end{aligned}
\end{align}
The second-order-cone characterization of $\cC_{H}(\Fset)$ will follow immediately. Note that $x_1+w^2x_2-2wX_{12}\ge 0$ for all $w>0$ if and only if $\begin{pmatrix}
    x_1 & -X_{12} \\ -X_{12} & x_2
\end{pmatrix}$ is in the copositive cone, which coincides with the two-dimensional semidefinite cone in the half space $-X_{12}\le 0$. Hence \eqref{eq:family-I-det} holds.
\end{proof}

\subsection{A semidefinite representation of \texorpdfstring{$\cC(\Fset)$}{C(F)}}\label{subsec:exact-sdr}

In order to characterize $\cC(\Fset)$, the family of bitangent rays is the only missing piece. 
We start by adding infinitely many bitangent cuts of form $G_{u,v}^1$ to $\cC_H(\Fset)$ and characterize the cone
\begin{align}
\cC_{H,G^{1}}(\Fset)=&\Set{\mxX\in\cC_{H}(\Fset)| G_{u,v}^{1}(\vx, X)\ge 0,\ \forall 0\leq u<v}.    
\end{align}
We will show in the proposition below that $\cC_{H,G^{1}}(\Fset)$ is semidefinte representable with one shadow variable.

\begin{proposition}
\label{prop:family-II-sdr}
We have
\begin{align}
    \label{eq:CHG1_sdr}
    \cC_{H,G^{1}}(\Fset)=&\Set{\mxX\in\cC_{H}(\Fset)| \exists s\text{ s.t. }\begin{pmatrix}x_2&X_{12}\\X_{12}&s\end{pmatrix}\succeq0,
\begin{pmatrix}
X_{11}&s&x_1-s\\
s&X_{12}&0\\
x_1-s&0&1-X_{12}
\end{pmatrix}\succeq0}.
\end{align}
\end{proposition}
\begin{proof}
Fix $\mxX\in\cC_H(\Fset)$. For one direction, suppose that there exists $s$ such that the linear matrix inequalities described in \eqref{eq:CHG1_sdr} holds. Then for all $v\ge u$ we have
\begin{align}
    & G^{1}_{u,v}(\vx, X) 
    \\
    = & [X_{11}  -2vs - 2u(x_1-s) + v^2X_{12} + (1-X_{12})u^2] + 2(v-u)[v^2x_2 - 2vX_{12} + s]
    \\
    = &
\begin{pmatrix}-1\\v\\u\end{pmatrix}^{\!T}
\begin{pmatrix}
X_{11}&s&x_1-s\\
s&X_{12}&0\\
x_1-s&0&1-X_{12}
\end{pmatrix}
\begin{pmatrix}-1\\v\\u\end{pmatrix}
+2(v-u)
\begin{pmatrix}v\\-1\end{pmatrix}^{\!T}
\begin{pmatrix}x_2&X_{12}\\X_{12}&s\end{pmatrix}
\begin{pmatrix}v\\-1\end{pmatrix}
\\
\ge & 0.
\end{align}
Conversely, suppose that
$G^1_{u,v}(\vx,X)\geq0$ for every $0\leq u<v$. 
By the Shor-RLT constraints and
Proposition~\ref{prop:family-I-soc}, we have
\begin{align}
\label{eq:tmp1}
0\leq X_{12}\leq1,\ x_1,x_2\geq0,\
x_1x_2 \ge X_{12}^2,\ \text{and } X_{11}\geq x_1^2.
\end{align}
We will consider the following different cases.

First, if $X_{12}=0$ or $X_{12}=1$, then the linear matrix inequalities in \eqref{eq:CHG1_sdr} are satisfied with $s=0$ and $s=x_1$ respectively due to \eqref{eq:tmp1}. 

Second, if $0<X_{12}<1$, then by \eqref{eq:tmp1} we have $x_2>0$. By Schur complements, the linear matrix inequality in \eqref{eq:CHG1_sdr} is equivalent to
\begin{align}
    \label{eq:tmp2}
    X_{11} \ge \frac{s^2}{X_{12}}+\frac{(x_1-s)^2}{1-X_{12}}\text{ and }s\ge \frac{X_{12}^2}{x_2}.
\end{align}
Note that if $x_1x_2\ge X_{12}$, then by \eqref{eq:tmp1} the above relation holds with $s = X_{12}x_1$. To finish the proof it suffices to study the case when $x_1x_2<X_{12}$. We will let
\begin{align}
    s = \frac{X_{12}^2}{x_2},\ u = \frac{x_1x_2-X_{12}^2}{x_2(1-X_{12})},\ \text{and }v = \frac{X_{12}}{x_2}.
\end{align}
Noting from $X_{12}^2\leq x_1x_2<X_{12}$ that
$
    u\ge 0$  and $v - u = {(X_{12}-x_1x_2)}/{(x_2(1-X_{12}))}>0.$
Therefore $G^1_{u,v}(\vx,X)\ge 0$. Here by direct computation we also have
\begin{align}
G^1_{u,v}(\vx,X) = 
X_{11}-\frac{s^2}{X_{12}}
-\frac{(x_1-s)^2}{1-X_{12}}.
\end{align}
Hence \eqref{eq:tmp2} holds and the linear matrix inequality in \eqref{eq:CHG1_sdr} is satisfied.
\end{proof}

A few remarks are in place. First, if $X_{11}=x_1^2$, $X_{12}=x_1x_2$, and $X_{22} = x_2^2$, then 
\begin{align}
\label{eq:s1s2}
    \det\begin{pmatrix}
 X_{11}&s&x_1-s\\
 s&X_{12}&0\\
 x_1-s&0&1-X_{12}
 \end{pmatrix} = -(s - x_1^2x_2)^2.
\end{align}
From this relation, we can observe that the auxiliary variable $s$ represents a higher-degree term $x_1^2x_2$. Indeed, the auxiliary variable $s$ has a natural interpretation from the moment matrix description of the Lasserre hierarchy \cite{Lasserre2001}. We will elaborate the relation in more detail later after Theorem \ref{thm:two-shadow-hull}. Second, apart from the Shor constraint $\mxX\succeq 0$, all other constraint in the description of \eqref{eq:CHG1_sdr} are second-order-cone representable. To see this, it suffices to observe that the last $3\times 3$ PSD constraint is second-order-cone representable:
\begin{align}
    \begin{pmatrix}
X_{11}&s&x_1-s\\
s&X_{12}&0\\
x_1-s&0&1-X_{12}
\end{pmatrix}\succeq 0\text{ if and only if }\exists \tilde s\text{ s.t. }\begin{pmatrix}X_{12}&s\\s&\tilde s\end{pmatrix}\succeq0
\text{ and }
\begin{pmatrix}1-X_{12}&x_1-s\\x_1-s&X_{11}-\tilde s\end{pmatrix}\succeq0.
\end{align}
The above relation is straightforward from a Schur-complement argument involving $1-X_{12}$.

We are now ready to describe a new semidefinite representation of $\cC(\Fset)$. 

\begin{theorem}\label{thm:two-shadow-hull}
    We have
    \begin{align}
        \cC(\Fset) = & \left\{\mxX\succeq 0\middle|   \begin{pmatrix}x_1&X_{12}\\X_{12}&x_2\end{pmatrix}\succeq 0, \text{ and } \exists s_1,s_2\in\mathbb{R} \text{ s.t. }\right.
         \\
 &\qquad\qquad\qquad\quad\begin{pmatrix}x_2&X_{12}\\X_{12}&s_1\end{pmatrix}\succeq0,
 \begin{pmatrix}
 X_{11}&s_1&x_1-s_1\\
 s_1&X_{12}&0\\
 x_1-s_1&0&1-X_{12}
 \end{pmatrix}\succeq0
         \\
     &\qquad\qquad\qquad\quad\left.\begin{pmatrix}x_1&X_{12}\\X_{12}&s_2\end{pmatrix}\succeq0,
 \begin{pmatrix}
 X_{22}&s_2&x_2-s_2\\
 s_2&X_{12}&0\\
 x_2-s_2&0&1-X_{12}
 \end{pmatrix}\succeq0\right\}.
    \end{align}
\end{theorem}
\begin{proof}
By symmetry, if we add all cuts of form $G^2_{u,v}(\vx,X)\ge 0$ to $\cC_H(\Fset)$ and consider the cone
\begin{align}
\cC_{H,G^{2}}(\Fset)=&\Set{\mxX\in\cC_{H}(\Fset)| G_{u,v}^{2}(\vx, X)\ge 0,\ \forall 0\leq u<v},
\end{align}
then similar to Proposition \ref{prop:family-II-sdr} we have
    \begin{align}
    \label{eq:CHG2_sdr}
    \cC_{H,G^{2}}(\Fset)=&\Set{\mxX\in\cC_{H}(\Fset)| \exists s\text{ s.t. }\begin{pmatrix}x_1&X_{12}\\X_{12}&s\end{pmatrix}\succeq0,
\begin{pmatrix}
X_{22}&s&x_2-s\\
s&X_{12}&0\\
x_2-s&0&1-X_{12}
\end{pmatrix}\succeq0}.
\end{align}
We conclude the theorem immediately by noting that $\cC(\Fset) = \cC_{H,G^{1}}(\Fset)\cap \cC_{H,G^{2}}(\Fset)$ and that the inequalities $x_1, x_2\ge 0$ and $0\le X_{12}\le 1$ are implied by the semidefinite constraints.
\end{proof}

A few remarks are in place. First, as discussed after Proposition \ref{prop:family-II-sdr}, all constraints describing $\cC(\Fset)$ other than Shor are second-order-cone representable with four auxiliary variables in total. Second, the variables $s_1$ and $s_2$ have a natural interpretation in the order-two (degree-four) Lasserre moment-SOS hierarchy with a redundant constraint $x_1x_2\geq0$ \cite{Lasserre2001} (see also \cite{lasserre2018moment}). Specifically, let \(\mathcal L_y\) denote the Riesz functional associated with a feasible truncated moment sequence $y$, so that $y_{ij}=\mathcal L_y(x_1^i x_2^j)$, and note that
$
y_{00}=1$, $y_{10}=x_1$, $y_{01}=x_2$,
$y_{20}=X_{11}$, $y_{11}=X_{12}$, and $y_{02}=X_{22}$.
This lift contains the degree-three moment
$
y_{21}=\mathcal L_y(x_1^2x_2).
$
Setting \(s_1:=y_{21}\), the localizing block for \(x_2\geq0\)
on the monomial basis \((1,x_1)\) is
\[
\mathcal L_y\left(
x_2
\begin{pmatrix}1\\x_1\end{pmatrix}
\begin{pmatrix}1&x_1\end{pmatrix}
\right)
=
\begin{pmatrix}
y_{01}&y_{11}\\
y_{11}&y_{21}
\end{pmatrix}
=
\begin{pmatrix}
x_2&X_{12}\\
X_{12}&s_1
\end{pmatrix}\succeq 0.
\]
The localizing terms for $x_1x_2\geq0$ and $1-x_1x_2\geq0$ also combine to give
\begin{align*}
&\mathcal L_y\left(
x_1x_2\,\begin{pmatrix}
	x_1\\1\\0
\end{pmatrix}\begin{pmatrix}
x_1\\1\\0
\end{pmatrix}^\top+(1-x_1x_2)\begin{pmatrix}
x_1\\0\\1
\end{pmatrix}\begin{pmatrix}
x_1\\0\\1
\end{pmatrix}^\top
\right)
\\
=&
\begin{pmatrix}
y_{20}&y_{21}&y_{10}-y_{21}\\
y_{21}&y_{11}&0\\
y_{10}-y_{21}&0&y_{00}-y_{11}
\end{pmatrix}
=
\begin{pmatrix}
X_{11}&s_1&x_1-s_1\\
s_1&X_{12}&0\\
x_1-s_1&0&1-X_{12}
\end{pmatrix}\succeq 0.
\end{align*}
Here the above matrix is a compression of two separate localizing matrices. Separately, the two separate matrices contain the degree-four moment $y_{31}$, which cancels after combination since
$x_1x_2x_1^2+(1-x_1x_2)x_1^2=x_1^2$.
Interchanging $x_1$ and $x_2$ gives
$s_2:=y_{12}=\mathcal L_y(x_1x_2^2)$ and the other two blocks in Theorem~\ref{thm:two-shadow-hull}.
Thus $s_1$ and $s_2$ are shadow variables for the cubic moments $x_1^2x_2$ and $x_1x_2^2$, respectively.

We finish this section with a characterization of the homogenized moment cone $\mathcal K(\Fset)$ defined in \eqref{eq:KF}. For compact sets, $\mathcal K(\Fset)$ is straightforward from homogenization of $\cC(\Fset)$. However, our set $\Fset$ is non-compact and one can approach infinity from the origin along either direction $(1,0)^\top$ or $(0,1)^\top$. We will utilize the characterization of $\mathcal K(\Fset)$ for unboundedness nonconvex set in \cite{JoyceYang2024} to derive the following corollary.

\begin{corollary}
\label{cor:homogenized-moment-cone}
We have
\begin{align}
\label{eq:KF_sdr}
\begin{aligned}
    \mathcal K(\Fset)
=\Bigg\{\begin{pmatrix}
\tau&\vx^\top\\
\vx&X
\end{pmatrix}\ \Bigg|\ &\begin{pmatrix}
\tau&\vx^\top\\
\vx&X
\end{pmatrix}\succeq0,
\quad
\begin{pmatrix}x_1&X_{12}\\X_{12}&x_2\end{pmatrix}\succeq0,
\quad \exists s_1,s_2\in\mathbb R\text{ such that}
\\
&\begin{pmatrix}x_2&X_{12}\\X_{12}&s_1\end{pmatrix}\succeq0,
\quad
\begin{pmatrix}
X_{11}&s_1&x_1-s_1\\
s_1&X_{12}&0\\
x_1-s_1&0&\tau-X_{12}
\end{pmatrix}\succeq0,
\\
&\begin{pmatrix}x_1&X_{12}\\X_{12}&s_2\end{pmatrix}\succeq0,
\quad
\begin{pmatrix}
X_{22}&s_2&x_2-s_2\\
s_2&X_{12}&0\\
x_2-s_2&0&\tau-X_{12}
\end{pmatrix}\succeq0
\Bigg\}.
\end{aligned}
\end{align}
\end{corollary}

\begin{proof}
Noting that the asymptotic cone of $F$ is exactly the two nonnegative axes, by 
\cite[Lemma~5]{JoyceYang2024} we have
$
\mathcal K(\Fset)=
\conv\{yy^\top:y\in\cone(\{1\}\times\Fset)
\cup\{(0,1,0)^\top, (0,0,1)^\top\}\}.$
To finish the proof it suffices to show that $yy^\top$ for both $y=(0,1,0)^\top$ and $y=(0,0,1)^\top$ can be represented in the form of $\begin{pmatrix}
    \tau  & \vx^\top\\\vx & X
\end{pmatrix}$ stated in \eqref{eq:KF_sdr}. In fact, the former can be represented with $\tau=s_1=s_2=x_1=x_2=X_{12}=X_{22}=0$ and $X_{11}=1$, and the latter can be represented with  $\tau=s_1=s_2=x_1=x_2=X_{12}=X_{11}=0$ and $X_{22}=1$.
\end{proof}

\section{Other interesting results}
\label{sec:poly-consequences}

In this section, we describe a few more interesting results that can be derived from Theorem \ref{thm:main}. Specifically, we show that the characterization of $\Ncone(\Fset)$ is also related to some topics in sum-of-squares (SOS) optimization.

\subsection{An SOS characterization of \texorpdfstring{$\Ncone(\Fset)$}{N(F)}}
Hilbert's classical result \cite{hilbert1888darstellung} already shows that  $\Sigma_{3,8}\subsetneq \Poct_{3,8}$, i.e., the cone of all SOS homogeneous ternary octics $\Sigma_{3,8}$ is a strict subset of the cone of all nonnegative homogeneous ternary octics $\Poct_{3,8}$. However, an exceptional linear section of $\Poct_{3,8}$ was identified in \cite{Harris1999} which is inside $\Sigma_{3,8}$. Specifically, the author proved in \cite{Harris1999} that all even symmetric nonnegative homogeneous ternary octics are SOS. Later it was shown in \cite{goel2016choi} that this case is exceptional among all even symmetric forms and no other case has this property. 

The exceptional result in \cite{Harris1999} shows that $\mathcal{L}\cap \Poct_{3,8} = \mathcal{L}\cap \Sigma_{3,8}$ where $\mathcal{L}$ is a subspace of dimension $4$. The proof is achieved by characterizing all the extreme rays of $\mathcal{L}\cap \Poct_{3,8}$ and showing that they are all SOS. Such technique greatly motivates our research, yielding our enumeration of all extreme rays stated in Theorem \ref{thm:main}. Interestingly, our main result yields an SOS characterization of $\Ncone(\Fset)$, which reveals a $6$-dimensional subspace $\mathcal H$ such that $\mathcal{H}\cap \Poct_{3,8}=\mathcal{H}\cap \Sigma_{3,8}$. Our subspace $\mathcal{H}$ seems to be non-trivial and to the best of our knowledge does not exhibit any symmetry properties.

We can construct the above special $\mathcal{H}$ subspace through the following reparameterization. Define the dense subset
\(
\Fset^\circ:=\{(x_1,x_2)\in\mathbb R^2:
x_1\geq0,\ x_2>0,\ x_1x_2\leq1\}.
\)
The map
\(
f_0:\mathbb R^2\setminus\{(0,0)\}\rightarrow\Fset^\circ\) defined by
$
f_0(y_1,y_2):=\left(y_1^2,{1}/{(y_1^2+y_2^2)}\right)
$
is surjective.  Indeed, for \((x_1,x_2)\in\Fset^\circ\), one may take
\begin{align}
\label{eq:uv}
y_1=\sqrt{x_1},\qquad
y_2=\sqrt{\frac{1}{x_2}-x_1}.
\end{align}
Since \(\overline{\Fset^\circ}=\Fset\), continuity then allows
nonnegativity on \(\Fset\) to be tested through this parametrization.

For a quadratic \(q_{R,\vr,p}\) as in \eqref{eq:q-def-intro}, let
\(
G_{R,\vr,p}(y_1,y_2):=(y_1^2+y_2^2)^2q_{R,\vr,p}\circ f_0(y_1,y_2).
\)
Homogenizing \(G_{R,\vr,p}\) with a new variable \(y_3\) gives the ternary
octic
\begin{equation}\label{eq:P-matrix}
\begin{aligned}
P_{R,\vr,p}(y_1,y_2,y_3)
&:=y_3^8G_{R,\vr,p}\left(\frac{y_1}{y_3},\frac{y_2}{y_3}\right)=\begin{pmatrix}A\\B\\C\end{pmatrix}^{\!T}
\begin{pmatrix}
R_{11}&r_1&R_{12}\\
r_1&p&r_2\\
R_{12}&r_2&R_{22}
\end{pmatrix}
\begin{pmatrix}A\\B\\C\end{pmatrix},
\end{aligned}
\end{equation}
where
\begin{equation}\label{eq:ABC}
A:=y_1^2(y_1^2+y_2^2),\qquad
B:=y_3^2(y_1^2+y_2^2),\qquad
C:=y_3^4.
\end{equation}

\begin{lemma}\label{lem:q-P}
The quadratic \(q_{R,\vr,p}\) is nonnegative on \(\Fset\) if and only if
\(P_{R,\vr,p}\) is nonnegative on \(\mathbb R^3\).
\end{lemma}

\begin{proof}
For \(y_3\neq0\) and \(y_1^2+y_2^2>0\),
\(
P_{R,\vr,p}(y_1,y_2,y_3)
=B^2q_{R,\vr,p}\left(A/B,C/B\right),
\)
\[
\frac AB=\frac{y_1^2}{y_3^2}\geq0,\ 
\frac CB=\frac{y_3^2}{y_1^2+y_2^2}\geq0,\text{ and } 
\frac{AC}{B^2}=\frac{y_1^2}{y_1^2+y_2^2}\leq1.
\]
Thus nonnegativity of \(q_{R,\vr,p}\) on \(\Fset\) implies nonnegativity of
\(P_{R,\vr,p}\) at these points.  The remaining points follow from
\(
P_{R,\vr,p}(0,0,y_3)=R_{22}y_3^8\) and \(
P_{R,\vr,p}(y_1,y_2,0)=R_{11}y_1^4(y_1^2+y_2^2)^2,
\)
where \(R_{11},R_{22}\geq0\) are forced by the two unbounded axes of
\(\Fset\).

Conversely, setting \(y_3=1\) gives
\(
P_{R,\vr,p}(y_1,y_2,1)
=(y_1^2+y_2^2)^2q_{R,\vr,p}\circ f_0(y_1,y_2).
\)
Surjectivity of \(f_0\) gives nonnegativity on \(\Fset^\circ\), and
continuity gives nonnegativity on \(\Fset\).
\end{proof}

By the above lemma, we obtain a sum-of-squares characterization of $\Ncone(\Fset)$: $q_{R,\vr,p}\in\Ncone(\Fset)$ if and only if the following ternary octic is a nonnegative polynomial:
\begin{align}
\label{eq:P-expanded}
\begin{aligned}
\Poct_{R,\vr,p}
={}&R_{11}y_1^8+2R_{11}y_1^6y_2^2+2r_1y_1^6y_3^2+R_{11}y_1^4y_2^4
+4r_1y_1^4y_2^2y_3^2\\
&+(2R_{12}+p)y_1^4y_3^4+2r_1y_1^2y_2^4y_3^2
+(2R_{12}+2p)y_1^2y_2^2y_3^4\\
&+2r_2y_1^2y_3^6+p y_2^4y_3^4+2r_2y_2^2y_3^6+R_{22}y_3^8.
\end{aligned}
\end{align}
Let us define 
\begin{align}
    \label{eq:PSigma}
    \Poct:=\Set{\Poct_{R,\vr,p}|\Poct_{R,\vr,p}\text{ is nonnegative}}\text{ and }\Sigmacone:=\Set{\Poct_{R,\vr,p}|\Poct_{R,\vr,p}\text{ is SOS}}.
\end{align}
We can observe from \eqref{eq:P-expanded} that $\Poct = \mathcal{H}\cap \Poct_{3,8}$ and $\Sigma = \mathcal{H}\cap \Sigma_{3,8}$, where $\mathcal H$ is a 6-dimensional subspace parameterized by $(R,r,p)$.
Clearly we have $\Sigma\subseteq \Poct$.
One consequence of Theorem \ref{thm:main} is that it allows us to enumerate all extreme rays of $\Ncone(\Fset)$ and show that $\Poct_{R,\vr,p}$ is SOS whenever $q_{R,\vr,p}$ is an extreme ray of $\Ncone(\Fset)$. Consequently, $\Poct=\Sigma$, as described in the following theorem.

\begin{theorem}
    \label{thm:PisSOS}
    For every extreme ray of $\Ncone(\Fset)$, the associated ternary octic $\Poct_{R,\vr,p}$ is SOS. Consequently, 
    \begin{align}
        \Poct=\Sigmacone.
    \end{align}
\end{theorem}

\begin{proof}
    Observe that the map from $q_{R,\vr,p}$ to $\Poct_{R,\vr,p}$ is a linear isomorphism between cones $\Ncone(\Fset)$ and $\Poct$. Therefore $q_{R,\vr,p}$ is an extreme ray of $\Ncone(\Fset)$ if and only if it is an extreme ray of $\Poct$. Observing the relation between $P_{R,\vr,p}$ and $q_{R,\vr,p}$ in the proof of Lemma \ref{lem:q-P}, we may set $x_1 = A/B$ and $x_2=C/B$ and multiply all the extreme-ray expressions in Theorem \ref{thm:main} by $B^2$ to obtain extreme rays of $\Poct$:
\begin{enumerate}
 \item \((aA+cB+bC)^2\), where the affine line
 \(ax_1+bx_2+c=0\) meets \(\operatorname{int}(\Fset)\);
 \item \(A^2\) and \(C^2\);
 \item \(AB\), \(AC\), and \(BC\);
 \item \(B^2-AC\);
 \item Lifted tangent rays:
 \[
 AB+w^2BC-2wAC,\qquad w>0;
 \]
 \item Bitangent rays:
 \begin{align*}
 (A-uB)^2+(v-u)C\bigl(2v^2B-(3v-u)A\bigr),&
 \qquad 0\leq u<v,\\
 (C-uB)^2+(v-u)A\bigl(2v^2B-(3v-u)C\bigr),&
 \qquad 0\leq u<v.
 \end{align*}
\end{enumerate}
Here the first three classes of extreme rays are clearly SOS from the definition of $A$, $B$, and $C$ in \eqref{eq:ABC}. In addition, direct computation with \eqref{eq:ABC} also yields 
\begin{align}
& B^2-AC = 
(y_1y_2y_3^2)^2+(y_2^2y_3^2)^2,
\\
& AB+w^2BC-2wAC = \bigl(y_1y_3(y_1^2+y_2^2)-wy_1y_3^3\bigr)^2+w^2y_2^2y_3^6,
\\
& (A-uB)^2+(v-u)C\bigl(2v^2B-(3v-u)A\bigr) 
\\
 = & 
\bigl(y_1^2(y_1^2+y_2^2)-u\,y_2^2y_3^2-v\,y_1^2y_3^2\bigr)^2
+2(v-u)\bigl(y_1y_3(y_1^2-vy_3^2)\bigr)^2
\\
&+2(v-u)\left(y_2y_3\left(y_1^2-\frac{u+3v}{4}y_3^2\right)\right)^2
+\frac{(7v+u)(v-u)^2}{8}(y_2y_3^3)^2, \text{ and }
\\
& (C-uB)^2+(v-u)A\bigl(2v^2B-(3v-u)C\bigr) 
\\
 = & 
\bigl(y_3^4-vy_1^2y_3^2-uy_2^2y_3^2\bigr)^2
+2(v-u)\bigl(y_1y_3\bigl(y_3^2-v(y_1^2+y_2^2)\bigr)\bigr)^2
+(v-u)^2\bigl(y_1y_2y_3^2\bigr)^2.
\end{align}
Thus all extreme rays of $\Poct$ are SOS.

The cone \(\Poct\) is
closed and pointed.  Intersecting it with an affine hyperplane defined by
an element of the interior of its dual yields a compact base (for a detailed argument, see the Proof of Theorem \ref{thm:main} at the end of the paper); the
finite-dimensional Krein--Milman theorem therefore expresses the cone as
the closed conic hull of its extreme rays.  Our computation above places every extreme ray of $\Poct$ in the closed convex
cone \(\Sigmacone\).  Therefore $\Poct\subseteq \Sigmacone$. The other inclusion $\Sigmacone\subseteq \Poct$ is trivial.

\end{proof}

\subsection{A bounded-degree preordering certificate for quadratics on \texorpdfstring{$\Fset$}{F}}

In addition to the above ternary octics SOS characterization, in this subsection we show an alternative SOS characterization of $\Ncone(\Fset)$. Observing the description of $\Fset$, consider the set
\begin{align}
    T:=\Set{\sum_{(e_1,e_2,e_3)\in \{0,1\}^3}\sigma_{e_1,e_2,e_3} x_1^{e_1}x_2^{e_2}(1-x_1x_2)^{e_3}|\sigma_{e_1,e_2,e_3} \text{ is SOS}}.
\end{align}
The set $T$ is called a preordering generated by constraints $x_1\ge 0$, $x_2\ge 0$, and $1-x_1x_2\ge 0$. Note that the degree of SOS polynomials $\sigma_{e_1,e_2,e_3}$ above can be arbitrarily high.
For any polynomial $f$, the condition $f\in T$ is sufficient for nonnegativity over $\Fset$. 
Note that such condition is not always necessary, and the degree bound of the SOS terms $\sigma_{e_1,e_2,e_3}$ are usually either unavailable or extremely high. For our non-compact set $F$, the necessary and sufficient condition is the Krivine–Stengle Positivstellensatz \cite{Krivine1964,Stengle1974}: a polynomial $f$ is nonnegative over $\Fset$ if and only if there exists $g_1,g_2\in T$ and integer $m$ such that $g_1 f = f^{2m} + g_2$. However, the degree bound of such certificate is dependent on $f$ (see, e.g., \cite{heijmans2026degree}).   

Due to the lack of uniform degree bounds, the above known SOS results for general polynomials are not applicable when we try to characterize a semidefinite representation of $\Ncone(\Fset)$. However, since we are focusing on quadratic polynomials nonnegative on $\Fset$, with the semidefinite representation of $\cC(\Fset)$ in Theorem \ref{thm:two-shadow-hull}, through duality we can 
obtain the semidefinite representation of $\Ncone(\Fset)$ as well as a bounded-degree preordering certificate.

\begin{theorem}
\label{thm:two-shadow-dual}
We have
\begin{align}
    \label{eq:two-shadow-dual}
    \Ncone(\Fset) = \Set{\begin{pmatrix}
    R & \vr^\top \\ \vr & p
\end{pmatrix}|\exists t_1,\ldots,t_{19} \text{ s.t. }\Lambda_1,\ldots,\Lambda_6\succeq 0}
\end{align}
where
\begin{align}
 \Lambda_1:=&
\begin{pmatrix}
p- t_{11}- t_{19}
&r_1-\frac{ t_1}{2}- t_8-\frac{ t_{12}}2
&r_2-\frac{ t_3}{2}-\frac{ t_4}{2}- t_{16}
\\[10pt]
r_1-\frac{ t_1}{2}- t_8-\frac{ t_{12}}2
&R_{11}- t_6
&\begin{gathered}
R_{12}- t_2- t_5- t_{13}\\[-2pt]
{}-\tfrac{1}{2} (t_9- t_{11}
+ t_{17}- t_{19})
\end{gathered}
\\[10pt]
r_2-\frac{ t_3}{2}-\frac{ t_4}{2}- t_{16}
&\begin{gathered}
R_{12}- t_2- t_5- t_{13}\\[-2pt]
{}-\tfrac{1}{2} (t_9- t_{11}
+ t_{17}- t_{19})
\end{gathered}
&R_{22}- t_{14}
\end{pmatrix},
\\
 \Lambda_2&:=
\begin{pmatrix}
 t_1& t_2\\
 t_2& t_3
\end{pmatrix},\ 
 \Lambda_3:=
\begin{pmatrix}
 t_4& t_5\\
 t_5&2( t_8- t_7)
\end{pmatrix},
\ \Lambda_4:=
\begin{pmatrix}
 t_6& t_7& t_8\\
 t_7& t_9& t_{10}\\
 t_8& t_{10}& t_{11}
\end{pmatrix},
\\
 \Lambda_5&:=
\begin{pmatrix}
 t_{12}& t_{13}\\
 t_{13}&2( t_{16}- t_{15})
\end{pmatrix},
\text{ and } 
 \Lambda_6:=
\begin{pmatrix}
 t_{14}& t_{15}& t_{16}\\
 t_{15}& t_{17}& t_{18}\\
 t_{16}& t_{18}& t_{19}
\end{pmatrix}.
\end{align}
Equivalently, $q_{R,\vr,p}\in \Ncone(\Fset)$ if and only if there exists $t_1,\ldots,t_{19}$ such that $\Lambda_1,\ldots,\Lambda_6\succeq 0$ and 
\begin{align}
    \label{eq:two-shadow-dual-monomials}
\begin{aligned}
    q_{R,\vr,p}(\vx) = & \begin{pmatrix}
        1 \\ \vx
    \end{pmatrix}^\top\Lambda_1\begin{pmatrix}
        1 \\ \vx
    \end{pmatrix} + x_1\begin{pmatrix}
        1 \\ x_2
    \end{pmatrix}^\top \Lambda_2 \begin{pmatrix}
        1 \\ x_2
    \end{pmatrix} + x_2(1-x_1x_2)\begin{pmatrix}
        0 \\ 1
    \end{pmatrix}^\top \Lambda_2 \begin{pmatrix}
        0 \\ 1
    \end{pmatrix}
    \\
    & + x_2\begin{pmatrix}
        1 \\ x_1
    \end{pmatrix}^\top \Lambda_3 \begin{pmatrix}
        1\\ x_1
    \end{pmatrix}  + x_1x_2\begin{pmatrix}
        x_1 \\ 1 \\ 0
    \end{pmatrix}^\top \Lambda_4\begin{pmatrix}
        x_1 \\ 1 \\ 0
    \end{pmatrix} + (1-x_1x_2)\begin{pmatrix}
        x_1 \\ 0 \\ 1
    \end{pmatrix}^\top \Lambda_4\begin{pmatrix}
        x_1 \\ 0 \\ 1
    \end{pmatrix}
    \\
    & + x_1\begin{pmatrix}
        1 \\ x_2
    \end{pmatrix}^\top \Lambda_5\begin{pmatrix}
        1 \\ x_2
    \end{pmatrix} + x_1x_2\begin{pmatrix}
        x_2 \\ 1 \\ 0
    \end{pmatrix}^\top \Lambda_6\begin{pmatrix}
        x_2 \\ 1 \\ 0
    \end{pmatrix} + (1-x_1x_2)\begin{pmatrix}
        x_2 \\ 0 \\ 1
    \end{pmatrix}^\top \Lambda_6\begin{pmatrix}
        x_2 \\ 0 \\ 1
    \end{pmatrix}.
\end{aligned}    
\end{align}
\end{theorem}

\begin{proof}
    The result \eqref{eq:two-shadow-dual} follows from direct dual cone computation based on the duality relationship between $\mathcal K(\Fset)$ (see Corollary \ref{cor:homogenized-moment-cone}) and $\Ncone(F)$. Applying the result \eqref{eq:two-shadow-dual}, to show the next equivalence it suffices to show that \eqref{eq:two-shadow-dual-monomials} holds for any $t_1,\ldots, t_{19}$, which can be done by direct computation. Note that here the expression \eqref{eq:two-shadow-dual-monomials} is developed by rank-1 and rank-2 decompositions of the 6 semidefinite matrices described in Theorem \ref{thm:two-shadow-hull} with $X_{11} = x_1^2$, $X_{12}=x_1x_2$, $X_{22} = x_2^2$, $s_1=x_1^2x_2$, and $s_2 = x_1x_2^2$ (see also the remark after Theorem \ref{thm:two-shadow-hull}).
\end{proof}

By the results above, the condition $q\in T$ is a necessary and sufficient condition for certifying nonnegativity over the non-compact set $\Fset$. Moreover, our certificate has bounded degree: the degree of $\sigma_{e_1,e_2,e_3}x_1^{e_1}x_2^{e_2}(1-x_1x_2)^{e_3}$ can be bounded by $4$ and all the SOS polynomials $\sigma_{e_1,e_2,e_3}$ involved are quadratic. It should be noted that a bounded-degree preordering certificate  could also be derived by first describing the polynomial $\Poct_{R,\vr,p}$ using Gram matrices and then revert the parameterization \eqref{eq:uv}. However, such certificate has a higher degree bound (which is 6) and hence is skipped.

\subsection{A bounded-degree certificate for certain quartics on the half-strip}
\label{subsec:strip}

In \cite{Marshall2010} a long-standing problem in algebraic geometry concerning non-compact semi-algebraic sets is settled. Specifically, it  shows that $f(z_1,z_2)$ is a nonnegative polynomial over the strip $[0,1]\times\mathbb{R}$ if and only if $f(z_1,z_2) = \sigma_1(z_1,z_2) + \sigma_{2}(z_1,z_2)z_1(1-z_1)$, where $\sigma_1$ and $\sigma_2$ are SOS. Extension to the half-strip $[0,1]\times[0,\infty)$ and other non-compact sets is studied in \cite{NguyenPowers2012}. However, the degree bounds of the strip and half-strip cases are rarely studied. To the best of our knowledge, the only degree bound in the literature is concerning the strip: in \cite{EscorcieloPerrucci2021} a degree bound is established when $f$ is quadratic with respect to $z_2$. In this subsection, we show that Theorem \ref{thm:two-shadow-dual} in the previous section yields an interesting bounded-degree certificate for a family of quartics over the half-strip.

We reparameterize our variables by $z_1=x_1x_2\in[0,1]$ and $z_2=x_1$ so the underlying set of interest becomes the half-strip.
For \(z_2>0\), we then have \(x_2=z_1/z_2\).  Clearing denominators gives
the polynomial
\begin{equation}
\label{eq:f_quartic}
\begin{aligned}
f_{R,\vr,p}(z_1,z_2)
&:=z_2^2q_{R,\vr,p}\left(z_2,\frac{z_1}{z_2}\right)\\
&=R_{11}z_2^4+2r_1z_2^3+pz_2^2+2R_{12}z_1z_2^2
  +2r_2z_1z_2+R_{22}z_1^2.
\end{aligned}
\end{equation}
In the following theorem, we show that for the family of quartic polynomials of form  \eqref{eq:f_quartic}, we have a bound-degree certificate of nonnegativity over the half-strip. Note that simply replacing $z_2$ by $z_2^2$ in the theorem would yield a bounded-degree certificate of nonnegativity for a family of octics over the strip. The bounded-degree result for such octics could not be achieved by \cite{EscorcieloPerrucci2021} since the polynomial is no longer quadratic with respect to $z_2$. We skip the description on this certificate on the strip since it follows straightforwardly from the theorem.

\begin{theorem}
\label{thm:half-strip-degree}
\(f_{R,\vr,p}\) is nonnegative on the half-strip if and only if there exists $t_1,\ldots,t_{19}$ such that $\Lambda_1,\ldots,\Lambda_6\succeq 0$ and 
\begin{align}
    \label{eq:qDecomp-half-strip}
    \begin{aligned}    
    f_{R,\vr,p}(z_1,z_2) ={}&
    \begin{pmatrix}
        z_2 \\ z_2^2 \\ z_1
    \end{pmatrix}^{\!\top}\Lambda_1
    \begin{pmatrix}
        z_2 \\ z_2^2 \\ z_1
    \end{pmatrix}
    +z_2\begin{pmatrix}
        z_2 \\ z_1
    \end{pmatrix}^{\!\top}\Lambda_2
    \begin{pmatrix}
        z_2 \\ z_1
    \end{pmatrix}
    +z_1z_2(1-z_1)\begin{pmatrix}
        0 \\ 1
    \end{pmatrix}^{\!\top}\Lambda_2
    \begin{pmatrix}
        0 \\ 1
    \end{pmatrix}
    \\
    &+z_1z_2\begin{pmatrix}
        1 \\ z_2
    \end{pmatrix}^{\!\top}\Lambda_3
    \begin{pmatrix}
        1 \\ z_2
    \end{pmatrix}
    +z_1\begin{pmatrix}
        z_2^2 \\ z_2 \\ 0
    \end{pmatrix}^{\!\top}\Lambda_4
    \begin{pmatrix}
        z_2^2 \\ z_2 \\ 0
    \end{pmatrix}
    +(1-z_1)\begin{pmatrix}
        z_2^2 \\ 0 \\ z_2
    \end{pmatrix}^{\!\top}\Lambda_4
    \begin{pmatrix}
        z_2^2 \\ 0 \\ z_2
    \end{pmatrix}
    \\
    &+z_2\begin{pmatrix}
        z_2 \\ z_1
    \end{pmatrix}^{\!\top}\Lambda_5
    \begin{pmatrix}
        z_2 \\ z_1
    \end{pmatrix}
    +z_1\begin{pmatrix}
        z_1 \\ z_2 \\ 0
    \end{pmatrix}^{\!\top}\Lambda_6
    \begin{pmatrix}
        z_1 \\ z_2 \\ 0
    \end{pmatrix}
    +(1-z_1)\begin{pmatrix}
        z_1 \\ 0 \\ z_2
    \end{pmatrix}^{\!\top}\Lambda_6
    \begin{pmatrix}
        z_1 \\ 0 \\ z_2
    \end{pmatrix},
    \end{aligned}
\end{align}
where $\Lambda_1,\ldots,\Lambda_6$ are defined in Theorem \ref{thm:two-shadow-dual}.
\end{theorem}

\begin{proof}
    We show first that \(f_{R,\vr,p}\) is nonnegative on the half-strip $[0,1]\times[0,\infty)$ if and only if $q_{R,\vr,p}\in\Ncone(\Fset)$. Suppose first \(q_{R,\vr,p}\) is nonnegative on \(\Fset\).  For any $(z_1,z_2)$ on the half-strip, if
\(z_2>0\), then \((z_2,z_1/z_2)\in\Fset\) and hence $f_{R,\vr,p}(z_1,z_2)
=z_2^2q_{R,\vr,p}\left(z_2,{z_1}/{z_2}\right)\ge 0$. Moreover, by Theorem \ref{thm:two-shadow-dual} we can observe that \(R_{22}\geq0\).  Hence when $z_2=0$ we also have $
f_{R,\vr,p}(z_1,0)=R_{22}z_1^2\geq0$.
Conversely, for any \((x_1,x_2)\in\Fset\) with \(x_1>0\), choosing
\(z_2=x_1\) and \(z_1=x_1x_2\) and noting that $(z_1,z_2)$ is on the half-strip, we have $
q_{R,\vr,p}(x_1,x_2)=f_{R,\vr,p}(z_1,z_2)/z_2^2\geq0. 
$
By continuity we also have $q_{R,\vr,p}(x_1,x_2)\ge 0$ at $x_2=0$. Hence $q_{R,\vr,p}\in\Ncone(\Fset)$.

If $f_{R,r,p}$ is nonnegative over the half-strip, then $q_{R,\vr,p}\in\Ncone(\Fset)$, and by Theorem \ref{thm:two-shadow-dual} we know that \eqref{eq:two-shadow-dual} holds, i.e., there exists $t_1,\ldots,t_{19}$ such that $\Lambda_1,\ldots,\Lambda_6\succeq0$. By direct computation we also know that the relation \eqref{eq:qDecomp-half-strip} holds for all $t_1,\ldots,t_{19}$. 
Moreover, if $\Lambda_1,\ldots,\Lambda_6$, it is clear from \eqref{eq:qDecomp-half-strip} that $f_{R,r,p}$ is nonnegative over the half-strip.
\end{proof}

It should be pointed out that the SOS terms appearing in \eqref{eq:qDecomp-half-strip} has degree bound of at most $4$. Since $f_{R,\vr,p}$ is already quartic, such degree bound is optimal.

\section{Extreme rays of \texorpdfstring{$\Ncone(\Fset)$}{N(F)}}\label{sec:NF}

In this section, we prove the major result Theorem~\ref{thm:main}. 
The proof is divided to three parts. In the first part, we show that the rays described in Theorem \ref{thm:main} are all extreme rays of $\Ncone(\Fset)$. Consequently, it suffices to show that any extreme ray of $\Ncone(\Fset)$ is a positive multiple of one listed in the theorem. 
In the second part, we show that any extreme ray $q_{R,\vr,p}\in\Ncone(\Fset)$ with positive semidefinite $R$ has to be a positive multiple of an interior affine-square ray, an axis-square ray, $x_1$, or $x_2$. The proof is based on the conic hull characterization of all convex quadratics $q_{R,\vr,p}$ that spans an extreme rays in $\Ncone(\Fset)$.
In the third part, we show that any extreme ray $q_{R,\vr,p}\in\Ncone(\Fset)$ with $R\not\succeq 0$ are either $x_1x_2$, $1-x_1x_2$, a lifted tangent ray, or a bitangent ray. The last part is the most challenging. Unlike the second part, due to nonconvexity we could not easily characterize the set of all nonconvex quadratics that spans an extreme ray. Interestingly, the challenge from nonconvexity also brings us an opportunity that never appeared in the convex case: the nonnegativity on $F$ is equivalent to that on $\partial F$ for nonconvex quadratics. In the non-PSD case, such boundary property along with extremality enforce great restrictions on where the zeros of an extreme ray could appear. Consequently, we are able to decompose the characterization task to a small number of scenarios. To the best of our knowledge, our strategy in the non-PSD case has not yet appeared in any literature on characterizing lifted convex hulls. We believe that this strategy, in its own right, merits further examination and may inform future research.

We first establish
membership and extremality of all listed rays and then prove exhaustiveness,
treating the PSD and non-PSD quadratic parts separately. Throughout, we
identify \(q_{R,\vr,p}\) with its coefficient tuple and, when no confusion
can arise, with the ray it generates.

\subsection{Membership and extremality}
In this subsection, we show that all rays listed in Theorem \ref{thm:main} are extreme rays of $\Ncone(\Fset)$. We start by showing that they all belong to $\Ncone(\Fset)$.

\begin{proposition}[Membership of the candidate rays]
\label{prop:membership}
Every polynomial listed in Theorem~\ref{thm:main} is nonnegative on
\(\Fset\).
\end{proposition}

\begin{proof}
The affine-square and axis-square polynomials are nonnegative on all of
\(\mathbb R^2\). The polynomials
$
x_1,x_2,x_1x_2$, and $1-x_1x_2
$
are nonnegative on \(\Fset\) by their respective defining inequalities. For lifted tangent rays $h_w$, let us fix any \(w>0\). For every \(\vx\in\Fset\), the
arithmetic--geometric mean inequality and the fact that \(0\leq x_1x_2\leq1\) give
$
x_1+w^2x_2\geq 2w\sqrt{x_1x_2}\geq 2wx_1x_2.
$
Hence \(h_w\) is nonnegative on \(\Fset\).
For bitangent rays, first consider $g_{u,v}^{1}$ with any fixed \(0\leq u<v\). If \(\vx\in\Fset\) and \(x_1=0\), then
$
g^{1}_{u,v}(0,x_2) = u^2+2v^2(v-u)x_2 \geq 0.
$
If $\vx\in \Fset$ and $x_1>0$, noting that 
$
x_2\in[0, x_1^{-1}]
$
and that \(g^{1}_{u,v}\) is affine in \(x_2\),
it suffices to check nonnegativity at the two endpoints of the interval. At \(x_2=0\), we have
$
g^{1}_{u,v}(x_1,0)=(x_1-u)^2 \geq 0.
$
At \(x_2=x_1^{-1}\), direct factorization gives
$
x_1\,g^{1}_{u,v}(x_1,x_1^{-1}) = (x_1-v)^2\bigl(x_1+2(v-u)\bigr) \geq 0.
$
Thus \(g^{1}_{u,v}\) is nonnegative on \(\Fset\). The claim for \(g^{2}_{u,v}\)
follows by exchanging \(x_1\) and \(x_2\) in the above argument.
\end{proof}

Next, we establish extremality of every rays of interest. Here we will repeatedly use the following observation: If
$
q=q_1+q_2$ where $q_1,q_2 \in \Ncone(\Fset),
$
then every zero \((x_1,x_2)\) of \(q\) in \(\Fset\) has to be a zero of both \(q_1\) and
\(q_2\), i.e., $q_1(x_1,x_2)=q_2(x_1,x_2)=0$. 
In the following proposition, we establish extremality of rays with elementary structure of zeros; their zeros on $\Fset$ form a line
segment, a coordinate axis, or the positive hyperbola.

\begin{proposition}[Extremality of the elementary rays]
\label{prop:elementary-extreme}
The following polynomials span extreme rays of \(\Ncone(\Fset)\):
\begin{enumerate}
    \item \(\ell^2\), where \(\ell\) is a nonconstant affine function
    whose zero set meets \(\operatorname{int}(\Fset)\);
    \item \(x_1^2\), \(x_2^2\), \(x_1\), \(x_2\), and \(x_1x_2\);
    \item \(1-x_1x_2\).
\end{enumerate}
\end{proposition}

\begin{proof}
Let
$
q=q_1+q_2$ where $q_1,q_2\in\Ncone(\Fset).
$
We will show that, if $q$ is any ray listed in the proposition, then $q_1$ and $q_2$ are both nonnegative multiples of $q$. 

For interior affine-square rays $\ell^2$, suppose that \(q=\ell^2\), and consider the line
$
L:=\{\vx \in \R^2 \mid \ell(\vx)=0\}.
$
Observe that the zeros of $\ell^2$ on $F$, denoted by \(I:=L\cap\operatorname{int}(\Fset)\), is a
nondegenerate relatively open line segment, and that each \(q_i\) vanishes on \(I\).
A quadratic polynomial that vanishes on a nondegenerate segment of the
line \(L\) is divisible by its defining affine linear polynomial.
Consequently,
$
q_i=\ell m_i
$
for some affine linear polynomial \(m_i\).
Note that \(m_i\) has to vanish on \(I\), since we can derive a contradiction if there would be
a point \(z\in I\) at which \(m_i(z)\neq0\), elaborated as follows.  In a sufficiently small relative
neighborhood of \(z\), contained in \(\operatorname{int}(\Fset)\), the
function \(m_i\) would retain a fixed nonzero sign, whereas \(\ell\)
changes sign across \(L\).  The product \(\ell m_i\) would then be
negative on one side of \(L\), contradicting \(q_i\geq0\) on \(\Fset\).
Thus affine linear polynomials \(m_i\) and $\ell$ both vanish on the nondegenerate segment \(I\). This implies that
$
m_i=\lambda_i\ell
$
for some scalar \(\lambda_i\) and hence
$
q_i=\lambda_i\ell^2
$ for $i\in\{1,2\}$.
Nonnegativity of \(q_i\) gives \(\lambda_i\geq0\), so the ray generated
by \(\ell^2\) is extreme.

For axis-square and coordinate rays, 
Suppose that
\(
q\in\{x_1^2,x_1,x_1x_2\}.
\)
Observe that each such \(q\) vanishes on the nonnegative \(x_2\)-axis; hence each
summand \(q_i\), $i\in\{1,2\}$, vanishes there as well. Thus \(q_i\) has to be of form \(
q_i(\vx)=x_1\bigl(a_i x_1+b_i x_2+c_i\bigr)
\)
for some scalars \(a_i,b_i,c_i\).
Nonnegativity on the positive \(x_1\)-axis and the positive hyperbola gives
\[
q_i(t,0)=t(a_it+c_i)\geq0,
\qquad
q_i(t,t^{-1})=a_it^2+b_i+c_it\geq0
\qquad(t>0).
\]
Dividing the first inequality by \(t\) and letting \(t\downarrow0\) and
\(t\to\infty\) gives \(c_i,a_i\geq0\), respectively. Letting
\(t\downarrow0\) in the second gives
\(
b_i\geq0.
\)
Thus every \(q_i\) belongs to
\(
\cone\{x_1^2,x_1x_2,x_1\}.
\)
Coefficient comparison proves extremality of each of the three
generators. Exchanging \(x_1\) and \(x_2\) proves the claims for
\(x_2^2\) and \(x_2\).

Finally, for the hyperbola ray, suppose that
$
q=1-x_1x_2.
$
Observe that \(q\) vanishes on the entire positive hyperbola
$
\{(t,t^{-1}):t>0\}
$
and hence $q_1$ and $q_2$ both vanish there.  For any $i\in \{1,2\}$, writing
\(
q_i(\vx)
=a_i x_1^2+b_i x_1x_2+c_i x_2^2+d_i x_1+e_i x_2+f_i,
\)
from the observation we have that 
\(
t^2q_i(t,t^{-1})
=a_i t^4+d_i t^3+(b_i+f_i)t^2+e_i t+c_i
\)
vanishes for every \(t>0\). Hence $t^2q_i(t,t^{-1})$ is the zero polynomial, i.e., 
$a_i=c_i=d_i=e_i=0$ and 
$b_i+f_i=0$.
It follows that
\(
q_i=-b_i(1-x_1x_2).
\)
Since
\(
-b_i=q_i(0,0)\geq0,
\)
$q_i$ is a nonnegative multiple of \(1-x_1x_2\).  Thus \(1-x_1x_2\)
also spans an extreme ray.
\end{proof}

The extremality of the lifted tangent and bitangent rays are more involved. We will prove the results in the following propositions.

\begin{proposition}[Extremality of the lifted tangent rays]
\label{prop:tangent-extreme}
For every \(w>0\), the polynomial
\(
h_w(\vx)=x_1+w^2x_2-2wx_1x_2
\)
spans an extreme ray of \(\Ncone(\Fset)\).
\end{proposition}

\begin{proof}
Suppose that
\(
h_w=q_1+q_2\) where \(q_1,q_2\in\Ncone(\Fset).
\)
We will first show that for $i\in\{1,2\}$, $q_i$ has to be of form
\(
q_i(\vx)=b_i x_1x_2+d_i x_1+e_i x_2.
\)
Let us write
\(
q_i(\vx)
=a_i x_1^2+b_i x_1x_2+c_i x_2^2+d_i x_1+e_i x_2+f_i.
\)
Observe that \(h_w(0,0)=0\) and hence $q_1$ and $q_2$ both vanish at the origin. Thus
\(
f_i=0.
\)
Also, on the two coordinate axes we have
\(
q_i(t,0)=a_it^2+d_it\geq0\) and \(
q_i(0,t)=c_it^2+e_it\geq0\) for all $
t\geq0$.
It follows that
\(
a_i,c_i,d_i,e_i\geq0.
\)
Because \(h_w\) has no \(x_1^2\)- or \(x_2^2\)-term, coefficient comparison
in \(h_w=q_1+q_2\) gives
\(
a_1+a_2=0\) and \(
c_1+c_2=0.
\)
Thus
\(
a_i=c_i=0.
\)
Therefore, $q_i$ has to be of form
\(
q_i(\vx)=b_i x_1x_2+d_i x_1+e_i x_2
\) where $d_i,e_i\ge 0$.

Next, we will show that $q_i$'s of the above form has to be nonnegative multiples of $h_w$. 
Observing that \(h_w\) vanishes at the hyperbola point
\(
(w,w^{-1}),
\)
each \(q_i\) vanishes there as well. Consider the restriction at the hyperbola
\(
\phi_i(t):=q_i(t,t^{-1})
=b_i+d_it+\frac{e_i}{t},
\)
which is nonnegative on \((0,\infty)\) and satisfies \(\phi_i(w)=b_i+d_iw+\frac{d_iw^2}{w}=0\).
Since \(w\) is an interior local minimizer of \(\phi_i\), we also have 
\(\phi_i'(w)=d_i - \frac{e_i}{w^2}=0.
\)
Solving the $b_i$ and $e_i$ from the inequalities $\phi(w)=\phi'(x)=0$ we have
\(
e_i=d_iw^2
\)
and
\(
b_i=-2wd_i.
\)
Consequently,
\(
q_i
=d_i\bigl(x_1+w^2x_2-2wx_1x_2\bigr)
=d_i h_w.
\)
Since \(d_i\geq0\), $q_i$ is a nonnegative multiple of \(h_w\).
Thus \(h_w\) spans an extreme ray.
\end{proof}

\begin{proposition}[Extremality of the bitangent rays]
\label{prop:bitangent-extreme}
For every \(0\leq u<v\), the polynomial
\(
g^{1}_{u,v}(\vx)
=(x_1-u)^2+(v-u)x_2\bigl(2v^2-(3v-u)x_1\bigr)
\)
spans an extreme ray of \(\Ncone(\Fset)\). The same is also
true for \(g^{2}_{u,v}\).
\end{proposition}

\begin{proof}
Suppose that
\(
g^{1}_{u,v}=q_1+q_2\) where \(
q_1,q_2\in\Ncone(\Fset).
\)
On the nonnegative \(x_1\)-axis we have
\(
g^{1}_{u,v}(t,0)=(t-u)^2.
\)
We first prove that, for each \(i\in\{1,2\}\),
\begin{equation}
q_i(t,0)=\lambda_i(t-u)^2
\qquad\text{for some }\lambda_i\geq0.
\label{eq:bitangent-axis-summands}
\end{equation}
In fact, if \(u>0\), then observing that $(u,0)$ is a zero of $g_{u,v}^1$ we have \(q_i(u,0)=0\). Thus \(u\) is an interior zero of the
nonnegative quadratic \(q_i(\,\cdot\,,0)\) and is therefore a double
root, implying \eqref{eq:bitangent-axis-summands}.
If \(u=0\), then we can observe that $g_{u,v}^1$ vanishes at the original and hence $q_1$ and $q_2$ also vanishes there. Thus $q_i(t,0)$ has to be of form
\(
q_i(t,0)=t(\alpha_i t+\beta_i).
\)
As $q_i(t,0)\ge 0$ for all $t\ge 0$ we have $\alpha_i,\beta_i\ge 0$. But $q_1(t,0) + q_2(t,0)  = q_{u,v}^1(t,0) = t^2$, so $\beta_1+\beta_0=0$ and consequently $\beta_1=\beta_2=0$. Hence $q_1$ and $q_2$ are both nonnegative multiple of $t^2$, proving \eqref{eq:bitangent-axis-summands}.

Based on \eqref{eq:bitangent-axis-summands}, we can continue to show that for any $i\in\{1,2\}$, the form of $q_i$ has to be
\(q_i(\vx)
=\lambda_i(x_1-u)^2+b_i x_1x_2+e_i x_2
\)
for some $\lambda_i$, $b_i$, and $e_i$, where $\lambda_i\ge 0$. 
To see this, let \(c_i\) denote the coefficient of \(x_2^2\) in \(q_i\).  Because
\(q_i\) is nonnegative on the unbounded nonnegative \(x_2\)-axis, one must
have \(c_i\geq0\). Since \(g^{1}_{u,v}\) has no \(x_2^2\)-term,
\(c_1+c_2=0\), and hence \(c_i=0\). Together with
\eqref{eq:bitangent-axis-summands}, this gives
the desired form \(
q_i(\vx)
=\lambda_i(x_1-u)^2+b_i x_1x_2+e_i x_2.
\)

To finish the proof it suffices to show that any $q_i$ of the above form is a nonnegative multiplier of $g_{u,v}^1$. Observe that \(g^{1}_{u,v}\) vanishes at
\(
(v,v^{-1}).
\)
and hence each \(q_i\) vanishes there as well.  Its restriction to the
hyperbola is
\(
\phi_i(t):=q_i(t,t^{-1})
=\lambda_i(t-u)^2+b_i+\frac{e_i}{t}.
\)
This function is nonnegative on \((0,\infty)\) and satisfies
\(\phi_i(v)=0\) at $v>0$. Solving $b_i$ and $e_i$ from the relation $\phi(v) = \phi'(v) = 0$ we have
\(
b_i=-\lambda_i(v-u)(3v-u).
\)
and
\(
e_i=2\lambda_i v^2(v-u).
\)
It follows that
\[
q_i(\vx)
=\lambda_i\!\left[
(x_1-u)^2-(v-u)(3v-u)x_1x_2+2v^2(v-u)x_2
\right]
=\lambda_i g^{1}_{u,v}(\vx).
\]
Since \(\lambda_i\geq0\), both summands are nonnegative multiples of
\(g^{1}_{u,v}\).  

We can now conclude that \(g^{1}_{u,v}\) spans an extreme ray. Exchanging \(x_1\) and \(x_2\) proves the corresponding statement for
\(g^{2}_{u,v}\).
\end{proof}

\subsection{The PSD case}
\label{subsec:psd-case}

In this subsection, we classify all the extreme rays $q_{R,\vr,p}\in\Ncone(\Fset)$ in which $R\succeq 0$. To accomplish the task, we will first classify all rays $q_{R,\vr,p}\in\Ncone(\Fset)$ with $R\succeq 0$, and then remove all the non-extreme rays among them. 

The classification for all rays with $R\succeq 0$ relies on the following convex geometry observation: for convex quadratics $q_{R,\vr,p}\in\Ncone(\Fset)$, its nonnegativity can be extended to one
of the two simpler supersets of $F$: $q_{R,\vr,p}$ would either be nonnegative on $G_1:=\{\vx\in\R^2 \mid x_1x_2 \leq 1\}$ or $G_2:=\R_+^2$. 

\begin{lemma}
\label{lem:convex-extension}
Suppose that \(q_{R,\vr,p}\in \Ncone(\Fset)\) satisfies \(R\succeq0\). Then \(q_{R,\vr,p}\) is
nonnegative on either \(\Gone\) or \(\Gtwo\).
\end{lemma}

\begin{proof}
Suppose by the contrary that there exists $q:=q_{R,\vx,p}\in \Ncone(\Fset)$ with $R\succeq 0$, but neither conclusion holds. Observing that
\(\Fset=\Gone\cap\Gtwo\), there must exist
\(
\vx^-\in\Gone\setminus\Gtwo\)
and
\(\vx^+\in\Gtwo\setminus\Gone
\)
such that \(q(\vx^-),q(\vx^+)<0\). In particular,
\(\vx^+\in\R_{++}^2\) and \(x_1^+x_2^+>1\), whereas
\(\vx^-\notin\R_+^2\).

We claim that there exists \(\bar\vx\in\Fset\) that lies on the line segment between \(\vx^+\) and \(\vx^-\). By the claim and the convexity of $q$, we have \(\bar\vx=t\vx^-+(1-t)\vx^+\) for some
\(t\in(0,1)\) and
$
q(\bar\vx)
\leq t q(\vx^-)+(1-t)q(\vx^+)<0,
$
contradicting the assumption that $q$ is nonnegative on \(\Fset\).

To prove the claim, follow the segment from \(\vx^+\) to \(\vx^-\). At its first exit from
\(\R_{++}^2\), one coordinate is zero. Along the preceding portion, the
coordinate product varies continuously from a value greater than \(1\)
to \(0\). Hence the segment contains a point
\(
\bar\vx\in\R_{++}^2\) that satisfies \(
\bar x_1\bar x_2=1.
\)
Thus \(\bar\vx\in\Fset\). 
\end{proof}

By the above convex geometry observation, we are ready to characterize all rays $q_{R,\vr,p}\in\Ncone(\Fset)$ with $R\succeq 0$.

\begin{lemma}
\label{lem:psd-conic-reduction}
Suppose that \(q_{R,\vr,p}\in \Ncone(\Fset)\) satisfies \(R\succeq0\), then \(q_{R,\vr,p}\) is a conic combination
of affine square rays and members of
\[
\{x_1,x_2,x_1x_2,1-x_1x_2\}.
\]
\end{lemma}

\begin{proof}
Fix any $q:=q_{R,\vx,p}\in \Ncone(\Fset)$ with $R\succeq 0$. By Lemma~\ref{lem:convex-extension}, \(q\) is nonnegative on \(\Gone\)
or on \(\Gtwo\).

Suppose first that $q$ is nonnegative on $G_1$. Noting that $G_1$ has non-empty interior, by the S-lemma (see, e.g., \cite{polik2007survey}), there exists \(\lambda\geq0\) such that the quadratic
\(
\sigma(\vx):=q(\vx)-\lambda(1-x_1x_2)
\)
is nonnegative everywhere. 
Every globally nonnegative quadratic is a sum
of squares of affine linear polynomials, so
\(
q(\vx)=\sigma(\vx)+\lambda(1-x_1x_2)
\)
is a conic combination of affine square rays and $1-x_1x_2$.

Now suppose that \(q\) is nonnegative on \(\Gtwo\). Observe that
\[
q(\vx) = \mathbf{z}^\top Q\mathbf{z},\text{ where } \mathbf{z}:= \begin{pmatrix}
    1\\ \vx
\end{pmatrix} \text{ and }Q:=
\begin{pmatrix}
p & \vr^\top\\
\vr & R
\end{pmatrix}.
\]
Here $Q$ is a copositive matrix: for \(y=(t,\mathbf u^\top)^\top\in\R_+^3\), if \(t>0\), then
$
\mathbf y^\top Q \mathbf y=t^2q(\mathbf u/t)\geq0,
$
whereas if \(t=0\), then
$
\mathbf y^\top Q \mathbf y=\mathbf u^\top R\mathbf u\geq0.
$
By \cite{diananda1962non}, $Q$ can be decomposed to $Q=S+N$ where \(S\succeq0\) and \(N\in\mathbb S^3\)
is entry-wise nonnegative. Therefore
\(
q(\vx)=\mathbf z^\top S\mathbf z+\mathbf z^\top N\mathbf z.
\)
Here the first quadratic $
\mathbf z^\top S\mathbf z$ is a conic combination of affine squares, while
$
\mathbf z^\top N\mathbf z
=
N_{11}+N_{22}x_1^2+N_{33}x_2^2
+2N_{12}x_1+2N_{13}x_2+2N_{23}x_1x_2
$
is a conic combination of axis-squares, hyperbola, and coordinate rays.
\end{proof}

With the results from the above lemma, we finish the characterization of extreme rays in the PSD case by ruling out all the non-extreme rays among the affine square rays.

\begin{proposition}[Extreme rays in the PSD case]
\label{prop:psd-extremes}
The extreme rays of $\Ncone(\Fset)$ whose quadratic part is positive
semidefinite are precisely the nonnegative multiple of either an interior affine-square ray, an axis-square ray, $x_1$, or $x_2$.
\end{proposition}

\begin{proof}
Each ray listed in this proposition clearly has a positive-semidefinite quadratic part; it also spans an extreme ray by the results in Proposition~\ref{prop:elementary-extreme}. 

Conversely, suppose that \(q:=q_{R,\vr,p}\) spans an extreme ray of
\(\Ncone(\Fset)\) and \(R\succeq0\). By Lemma~\ref{lem:psd-conic-reduction}, \(q\) is a
conic combination of affine squares and members of
\(
\{x_1,x_2,x_1x_2,1-x_1x_2\}.
\)
All rays forming the conic combination clearly belong to $\Ncone(\Fset)$. 
Extremality and convexity of $q$ forces it to be a nonnegative multiple of an affine square ray, $x_1$, or $x_2$.

It remains to determine which affine squares can be extreme. By Proposition \ref{prop:elementary-extreme} we already know that interior affine square rays are extreme rays; it suffices to study on the ones whose zeros that do not meet $\operatorname{int}(F)$. Let
\[
\ell(\vx)=\alpha x_1+\beta x_2+\gamma
\]
be nonzero, and suppose that its zero set does not meet
\(\operatorname{int}(\Fset)\). Since \(\operatorname{int}(\Fset)\) is
connected, \(\ell\) has a constant sign in there. Replacing \(\ell\) by
\(-\ell\) if necessary and noting the fact that 
\(\Fset=\operatorname{cl}(\operatorname{int}(\Fset))\), we may assume that 
\[
\ell(\vx)\geq0
\qquad\forall\,\vx\in\Fset.
\]
Evaluating \(\ell\) on the two unbounded coordinate axes gives $\alpha,\beta,\gamma\ge 0$. 
Consequently, we have the following linear combination description of the affine square ray $\ell^2$:
\[
\ell(\vx)^2 =
\alpha^2x_1^2+\beta^2x_2^2+\gamma^2
+2\alpha\beta x_1x_2
+2\alpha\gamma x_1
+2\beta\gamma x_2.
\]
Every term on the right is nonnegative on \(\Fset\). If at least two of
\(\alpha,\beta,\gamma\) are nonzero, separating one of the nonzero
square terms from the remainder gives a nontrivial conic decomposition
of \(\ell^2\).
If exactly one coefficient is nonzero, then \(\ell^2\) is proportional to
\(x_1^2\), \(x_2^2\), or \(1\). The first two are the axis-square rays
listed. The constant polynomial is not extreme because
\(
1=(1-x_1x_2)+x_1x_2.
\)
Thus the only extreme affine squares whose zero sets do not meet
\(\operatorname{int}(\Fset)\) are \(x_1^2\) and \(x_2^2\).
\end{proof}

\subsection{The non-PSD case} \label{subsec:non-psd}

In this section we classify all the extreme rays $q_{R,\vr,p}\in \Ncone(\Fset)$ in which $R\not\succeq 0$. Unlike the PSD case in which we can start by characterizing the conic hull of all convex rays in $\Ncone(\Fset)$, the rays $q_{R,\vr,p}\in \Ncone(\Fset)$ with $R\not\succeq 0$ no longer form a conic hull. Therefore, the strategy used throughout the previous subsection could not be repeated. Luckily, the nonconvexity of rays $q_{R,\vr,p}$ in $\Ncone(\Fset)$ has an interesting boundary property that is not possessed by convex ones: any nonconvex $q_{R,\vr,p}$ is nonnegative on $F$ if and only if it is nonnegative on the boundary $\partial F$. In this subsection, we will rely on this crucial property to finish the characterization all of extreme rays in the non-PSD case. Specifically, we will show that the boundary property allows us to decompose the extreme ray classification task to a small number of different scenarios, depending on where the zero of the extreme ray is on the boundary.

\subsubsection{Scenario decomposition and boundary property}

We will decompose the extreme ray classification task to a small number of different scenarios. We start with an observation of any $q_{R,\vr,p}\in \Ncone(\Fset)$. Restricting on either the $x_1$- or $x_2$- axis we have $q_{R,\vr,p}(t,0) = R_{11}t^2+2r_1t+p$ and $q_{R,\vr,p}(0,t)
=R_{22}t^2+2r_2t+p.$
Using the fact that a quadratic \(at^2+2bt+c\) is nonnegative on \([0,\infty)\) if and
only if
$
a\geq0$,
 $
c\geq0,$ and 
$b\geq-\sqrt{ac}$, 
we have
\begin{align}
\label{eq:non-PSD-cond}
R_{11},R_{22},p\geq0,
\qquad
r_1\geq-\sqrt{pR_{11}},
\qquad
r_2\geq-\sqrt{pR_{22}}.
\end{align}
If $R\not\succeq 0 $, then $R_{11}R_{22} < R_{12}^2 $. based on the above observation, we have either $R_{12}>\sqrt{R_{11}R_{22}}$ or $R_{12}<-\sqrt{R_{11}R_{22}}$. 
We will show that the former case of $R_{12}$ yields a simple scenario of extreme ray classification, as stated below. 

\begin{proposition}[The simple scenario]
\label{prop:positive-offdiagonal}
Let \(q=q_{R,\vr,p}\) span an extreme ray of \(\Ncone(\Fset)\).  If 
\(R\not\succeq0\) and
\(
R_{12}>\sqrt{R_{11}R_{22}},
\)
then \(q\) is a positive multiple of \(x_1x_2\).
\end{proposition}

\begin{proof}
Attempting a completing the square on $q$, we have
\begin{align*}
q(\vx)
={}&s(\vx)^2
+2\bigl(r_1+\sqrt{pR_{11}}\bigr)x_1
+2\bigl(r_2+\sqrt{pR_{22}}\bigr)x_2\\
&+2\bigl(R_{12}-\sqrt{R_{11}R_{22}}\bigr)x_1x_2,
\end{align*}
where the complete square is $
s(\vx):=\sqrt p-\sqrt{R_{11}}x_1-\sqrt{R_{22}}x_2.
$
In the above completing-square equality, Every term at the right belongs to \(\Ncone(\Fset)\), every coefficient
is nonnegative due to \eqref{eq:non-PSD-cond}, and the coefficient of \(x_1x_2\) is strictly positive.
Extremality of $q$ then forces it to be only a positive multiple of $x_1x_2$.
\end{proof}

It remains to consider the case when 
\begin{equation}
R_{12}<-\sqrt{R_{11}R_{22}}.
\label{eq:negative-branch}
\end{equation}
For this case, the following boundary property will be crucial for us to decompose the extreme ray classification task to a small number of scenarios.

\begin{lemma}[Boundary property]
\label{lem:boundary-reduction}
Suppose that \(R\not\succeq0\). Then $q_{R,\vr,p}$ is nonnegative on $F$ if and only if it is nonnegative on $\partial F$.
\end{lemma}

\begin{proof}
Only the reverse implication requires proof. Consider any fixed $q:=q_{R,\vr,p}$ that is nonnegative on $\partial F$ with $R\not\succeq 0$. Suppose by the contrary that \(q(\bar\vx)<0\) for some \(\bar\vx\in\operatorname{int}(\Fset)\). Since $R\not\succeq 0$, there exists 
 \(\mathbf d\neq0\) such
that \(\mathbf d^\top R\mathbf d<0\). 
Consider the connected component
\[
\{\alpha\in\mathbb R:\bar\vx+\alpha \mathbf d\in\Fset\}.
\]
This is clearly a bounded interval of from \([\alpha_-,\alpha_+]\) with
\(\alpha_-<0<\alpha_+\) and both endpoints \(\bar\vx + \alpha_\pm \mathbf d\in \partial\Fset\). Restricting on this interval, $q$ is concave, hence its minimum is achieved at an endpoint. Since $q(\bar\vx)<0$, at the endpoints either $q(\bar\vx + \alpha_- \mathbf d)<0 $ or $q(\bar\vx + \alpha_+ \mathbf d)<0$, contradicting the fact that $q$ is nonnegative on $\partial F$.
\end{proof}

With the help of the above boundary property lemma, we are ready to classify the extreme rays whose quadratic part is not positive
semidefinite and satisfy \eqref{eq:negative-branch}. A terminology and a few notations will be used in the sequel. We call \(\bar\vx\in\partial\Fset\) a \emph{contact}
of \(q\) if \(q(\bar\vx)=0\). Denote the union of both nonnegative axes by $\mathcal A$ and the hyperbola portion of the boundary by $\mathcal H$, i.e.,
\[
\mathcal A
:=
\{(t,0):t\geq0\}\cup\{(0,t):t\geq0\},
\qquad
\mathcal H
:=
\{(t,t^{-1}):t>0\},
\]
so that \(\partial\Fset=\mathcal A\cup\mathcal H\).  The restrictions of
\(q=q_{R,\vr,p}\) to these boundary components are encoded by
\begin{align}
A^1_q(t)
&:=q(t,0)
=R_{11}t^2+2r_1t+p,
\label{eq:Ax}\\
A^2_q(t)
&:=q(0,t)
=R_{22}t^2+2r_2t+p,
\label{eq:Ay}\\
P_q(t)
&:=t^2q(t,t^{-1})
=R_{11}t^4+2r_1t^3+(2R_{12}+p)t^2+2r_2t+R_{22}.
\label{eq:P}
\end{align}
Thus \(q\geq0\) on \(\partial\Fset\) if and only if
\[
A^1_q(t),A^2_q(t)\geq0\quad(t\geq0),
\qquad
P_q(t)\geq0\quad(t>0).
\]
In the following proposition we describe a second simple scenario of extreme rays, in which $q$ has no contact on any axis. We rely on the boundary property lemma for this classification.

\begin{proposition}[The scenario without axis contacts]
\label{prop:no-axis-contact}
Let \(q=q_{R,\vr,p}\) span an extreme ray of \(\Ncone(\Fset)\) and
satisfy \eqref{eq:negative-branch}.  If \(q\) has no contact on
\(\mathcal A\), then \(q\) is a positive multiple of
\(1-x_1x_2\).
\end{proposition}

\begin{proof}
Suppose that $q$ has no zero on $\mathcal A$. 
Noting \eqref{eq:non-PSD-cond} and the fact that $q\in\Ncone(\Fset)$ we can observe that $\min_{\vx\in \mathcal A}q(\vx)$ exists and is positive.
Hence, for all sufficiently small
\(\varepsilon>0\),
\[
q_\pm:=q\pm\varepsilon(1-x_1x_2)
\]
are nonnegative on \(\mathcal A\).  On \(\mathcal H\), the
perturbation vanishes, so \(q_\pm=q\).  By taking \(\varepsilon\)
sufficiently small, the quadratic parts of both perturbations still
satisfy \eqref{eq:negative-branch}.  Lemma~\ref{lem:boundary-reduction}
therefore gives \(q_\pm\in\Ncone(\Fset)\).

Since \(q=(q_++q_-)/2\), extremality implies that \(q_+\) and \(q_-\)
are proportional to \(q\).  Their difference then shows that
\(1-x_1x_2\) is proportional to \(q\), and the proportionality factor
is positive.
\end{proof}

With the help of the boundary condition, we can also show that if an extreme ray $q$ has  contact, the number of contacts is extreme small, and the local of the contacts are limited. In the lemma below, we will show that $q$ has at most one contact on any axis. Note that we do not even need $q$ to be an extreme ray for the following lemma to hold.
\begin{lemma}[At most one axis contact]
\label{lem:at-most-one-axis-contact}
If \(q=q_{R,\vr,p}\in\Ncone(\Fset)\) satisfies
\eqref{eq:negative-branch}, then $q$ has at most one zero on $\mathcal A$.
\end{lemma}

\begin{proof}
We start by showing that $q$ would not have two distinct zeros on the nonnegative $x_1$-axis. 
Consider the restriction $A_q^1(t)$ defined in \eqref{eq:Ax}. Noting that $A_q^1(t)\ge 0$ for all $t\ge 0$, the existence of two distinct zeros $t_1,t_2>0$ with $A_q^1(t_1) = A_q^1(t_2)=0$ would force $A_q^1(t)\equiv 0$. Hence \(R_{11}=r_1=p=0\). However, restricting on the hyperbola we will then have
$\lim_{\to\infty}q(t,t^{-1})
=
2R_{12}.
$ Since $R_{12}<0$ by \eqref{eq:negative-branch}, there exists a point on the hyperbola at which $q$ is negative, contradicting the assumption that $q\in \Ncone(\Fset)$. By symmetry $q$ would also not have two distinct zeros on the nonnegative $x_2$-axis either.  

It remains to rule out the possibility of two zeros \((\xi,0)\) and \((0,\eta)\) of $q$ with
\(\xi,\eta>0\).  Such zeros are double zeros of the respective
axis restrictions, so the restriction $A_q^1(t)$ and $A_q^2(t)$ defined in \eqref{eq:Ax} and \eqref{eq:Ay} are 
$
A_q^1(t)=R_{11}(t-\xi)^2
$ and 
$A_q^2(t)=R_{22}(t-\eta)^2
$ respectively, 
where \(R_{11},R_{22}>0\).  In particular,
\(
r_2=-R_{22}\eta\) and
\(
R_{11}\xi^2=R_{22}\eta^2.
\)
Noting that \((\xi,s)\in\Fset\) for all sufficiently small \(s\geq0\), we can observe from $q(\xi,0)=0$ and $q(\xi,s)\ge 0$ that
\[
0\leq\partial_{x_2}q(\xi,0)
=2(R_{12}\xi+r_2).
\]
Consequently,
\(
R_{12}
\geq{R_{22}\eta}/{\xi}
=\sqrt{R_{11}R_{22}},
\)
contradicting \eqref{eq:negative-branch}.
\end{proof}

In the following lemma, we will further restrict the number of contacts on the hyperbola $\mathcal H$ through the boundary property lemma.

\begin{lemma}[At most one hyperbola contact]
\label{lem:at-most-one-hyperbola-contact}
Let \(q=q_{R,\vr,p}\) span an extreme ray of \(\Ncone(\Fset)\) and
satisfy \eqref{eq:negative-branch}.  If \(q\) is not a positive
multiple of \(1-x_1x_2\), then it has at most one contact on $\mathcal H$.
\end{lemma}

\begin{proof}
Consider any extreme ray $q = q_{R,\vr,p}\in\Ncone(\Fset)$ that satisfies \eqref{eq:negative-branch} and is not a positive multiple of $1-x_1x_2$. We will study 
$P_q(t)$ defined in \eqref{eq:P}, which is the restrict of $q$ on $\mathcal H$. 
Suppose by the contrary that \(\phi\) has two distinct zeros
\((u,u^{-1})\) and \((v,v^{-1})\) on \(\mathcal H\).  Since $q\in\Ncone(\Fset)$, the
corresponding positive roots of \(P_q\) have even multiplicity; hence,
for some \(\lambda\ge 0\),
\[
P_q(t)=\lambda(t-u)^2(t-v)^2.
\]
Here the $\lambda=0$ case is included, in which there are infinitely many zeros on the hyperbola. However, observe that $(t-u)^2(t-v)^2$ is also the restriction of an affine square ray over the hyperbola: 
for
\[
\ell_{u,v}(x_1,x_2):=x_1-(u+v)+uvx_2
\]
we have
\(
t^2\ell_{u,v}(t,t^{-1})^2=(t-u)^2(t-v)^2.
\)
It follows, by coefficient comparison, that
\(
q=\lambda\ell_{u,v}^2+\delta(1-x_1x_2)
\)
for some \(\delta\in\mathbb R\).  The quadratic coefficients satisfy
\(
R_{11}=\lambda,\)
\(
R_{22}=\lambda u^2v^2,\) and
\(
R_{12}=\lambda uv-{\delta}/{2}.
\)
Condition \eqref{eq:negative-branch} therefore yields
\(
\lambda uv-{\delta}/{2}
<-\lambda uv,\) so \(
\delta>4\lambda uv \ge 0.
\)
Thus \(q\) is the conic combination of two non-proportional members
\(\ell_{u,v}^2\) and \(\delta(1-x_1x_2)\) of
\(\Ncone(\Fset)\), while the coefficient $\delta$ of \(1-x_1x_2\) is strictly positive. Extremality of $q$ then forces it to be a positive multiple of \(1-x_1x_2\), a contradition.
\end{proof}

The preceding results and the boundary property lemma now reduce the remaining cases of extreme rays to a specialized mixed contact pattern.

\begin{lemma}[The remaining scenarios]
\label{lem:negative-contact-pattern}
\label{lem:unique-contact}
Let \(q=q_{R,\vr,p}\) span an extreme ray of \(\Ncone(\Fset)\) and
satisfy \eqref{eq:negative-branch}.  Then exactly one of the following
holds:
\begin{enumerate}
\item \(q\) is a positive multiple of \(1-x_1x_2\);
\item \(q\) has exactly two zeros on \(\partial\Fset\), one on
      \(\mathcal A\) and one on \(\mathcal H\).
\end{enumerate}
\end{lemma}

\begin{proof}
Assume that the first alternative does not hold.  By
Proposition~\ref{prop:no-axis-contact} and
Lemma~\ref{lem:at-most-one-axis-contact}, \(q\) has exactly one axis
contact.  By Lemma~\ref{lem:at-most-one-hyperbola-contact}, it has at
most one hyperbola contact.  It remains only to prove that such a
contact exists.

Suppose otherwise, and set 
\[\phi(t):=q(t,t^{-1})=R_{11}t^2+2r_1t+(2R_{12}+p)+\frac{2r_2}{t}+\frac{R_{22}}{t^2}.\]  
Then
\(\phi>0\) on \((0,\infty)\).  Moreover, we can observe that
$
\phi(t)\to+\infty$ as $t\to0^+$ and as $t\to\infty$, as detailed in the following arguments. 
If the axis contact is the origin, then \(p=0\), and by \eqref{eq:non-PSD-cond} all four
coefficients \(R_{11},R_{22},r_1,r_2\) are nonnegative, and neither
pair \((R_{11},r_1)\), \((R_{22},r_2)\) can be zero because the axis
restrictions are nonzero.  If the contact is \((u,0)\) with \(u>0\),
then \(A^1_q(t)=R_{11}(t-u)^2\) with \(R_{11}>0\), which gives the
limit at infinity.  The pair \((R_{22},r_2)\) cannot vanish, since
otherwise \(q(u,u^{-1})=2R_{12}<0\); this gives the limit at zero.
The case of a contact on the open \(x_2\)-axis is symmetric.

Due to the above properties, \(m:=\min_{t>0}\phi(t)>0\).  For sufficiently small
\(\varepsilon>0\), the polynomials
$
q_\pm:=q\pm\varepsilon x_1x_2
$
agree with \(q\) on \(\mathcal A\), are nonnegative on
\(\mathcal H\), and still satisfy \eqref{eq:negative-branch}.
Lemma~\ref{lem:boundary-reduction} gives
\(q_\pm\in\Ncone(\Fset)\).  This is a nontrivial decomposition of
\(q=(q_++q_-)/2\): the perturbation \(x_1x_2\) vanishes on all of
\(\mathcal A\), whereas \(q\) does not.  This contradicts extremality.
Thus the hyperbola contact exists.
\end{proof}

Based on the above lemma, there are only two remaining scenarios left when classifying extreme ray $q$. The first scenario is when $q$ has exactly one contact on a positive axis and one other on the hyperbola. We call it the open-axis and hyperbola contact scenario. The second is when $q$ has exactly one contact at the origin and one other on the hyperbola, which we call the origin and hyperbola contact scenario. The characterization of extreme rays for the remaining two scenarios are studied below.

\subsubsection{Remaining non-PSD scenarios and overall extreme ray characterization}
We are now ready to study the only two remaining scenarios and finish the proof of our major result in Theorem \ref{thm:main}. We start with the open-axis and hyperbola contact scenario.

\begin{proposition}[The open-axis and hyperbola contact scenario]
\label{prop:open-axis-hyperbola}
Let \(q=q_{R,\vr,p}\) span an extreme ray of \(\Ncone(\Fset)\) and
satisfy \eqref{eq:negative-branch}. If
$q(u,0)=q(v,v^{-1})=0$ for some $u,v>0$, 
then \(0<u<v\) and \(q\) is a positive multiple of
\(g^1_{u,v}\). 
If
$
q(0,u)=q(v^{-1},v)=0$ for some $u,v>0$,
then \(0<u<v\) and \(q\) is a positive multiple of
\(g^2_{u,v}\).
\end{proposition}

\begin{proof}
We will only prove the statement involving $g_{u,v}^1$; the one involving $g_{u,v}^2$ follows by exchanging
\(x_1\) and \(x_2\).  The zero \(u\) is an interior zero of the
nonnegative, nonzero quadratic \(A^1_q\) defined in \eqref{eq:Ax}.  Hence
\(
A^1_q(t)=\lambda(t-u)^2
\)
for some \(\lambda>0\).  After scaling, without of generality we may assume \(\lambda=1\) in the description of the extreme $q$:
\[
q(\vx)
=(x_1-u)^2+R_{22}x_2^2+2R_{12}x_1x_2+2r_2x_2.
\]
The contact at \((v,v^{-1})\) implies that $t=v$ is a double positive root of the monic
quartic \(P_q\) defined in \eqref{eq:P}. Noting that its cubic coefficient is \(-2u\), so its form has to be
\begin{equation}
P_q(t)
=(t-v)^2\bigl((t+d)^2-\gamma\bigr).
\label{eq:xh-factor}
\end{equation}
for some
\(\gamma\in\mathbb R\), where we denote $d:=v-u$.
Comparing coefficients gives
\begin{align}
\label{eq:q_coeffs}
R_{22}=v^2\bigl(d^2-\gamma\bigr),\ 
R_{12}=-vd-\frac{\gamma}{2},\text{ and }
r_2=uvd+v\gamma.
\end{align}

We can observe that \(0 < \gamma \leq d^2\). Indeed, if \(\gamma \leq 0\), by \eqref{eq:q_coeffs}
$
R_{12}
\geq-vd
\geq-v\sqrt{d^2-\gamma}
=-\sqrt{R_{11}R_{22}},
$
contradicting \eqref{eq:negative-branch}; if \(\gamma > d^2\), then \(P_q(0) < 0\) contradicting the non-negativity of \(P_q\). Note also that \(d>0\); otherwise \(P_q(-d) = -(t-v)^2\gamma <0\).  

With $d>0$, by directly computation based on the relation \eqref{eq:q_coeffs}, we have
\begin{equation}
q(\vx)
=\left(1-\frac{\gamma}{d^2}\right)
  (x_1-u-vd\,x_2)^2
+\frac{\gamma}{d^2}g^1_{u,v}(\vx).
\label{eq:xh-two-ray-decomposition}
\end{equation}
Hence by the arguments above, we have \(1-\frac{\gamma}{d^2} \geq 0\) and \(\frac{\gamma}{d^2} > 0\). Therefore $q$ is a conic combination of an affine square ray and a bitangent ray $g^1_{u,v}$, and the coefficient of $g^1_{u,v}$ is strictly positive. Extremality of $q$ then requires that it is a positive multiple of \(q=g^1_{u,v}\).
\end{proof}

Our last scenario of study is the origin and hyperbola contact scenario. 

\begin{proposition}[The origin and hyperbola contact scenario: a direct proof]
\label{prop:corner-hyperbola}
Let \(q=q_{R,\vr,p}\) span an extreme ray of \(\Ncone(\Fset)\) and
satisfy \eqref{eq:negative-branch}. 
If
\(
q(0,0)=q(v,v^{-1})=0\)
for some \(v>0,
\)
then \(q\) is a positive multiple of one of the following lifted tangent or bitangent rays: 
$g^1_{0,v}$,
$h_v$, or
$g^2_{0,1/v}$.
\end{proposition}

\begin{proof}
Since there is a contact at the origin, $p=0$.  It follows from
\eqref{eq:non-PSD-cond} that
$R_{11},R_{22},r_1,r_2\geq0$.  Moreover, $P_q(t)\geq0$ for
$t>0$ and $P_q(v)=0$, so $v$ is a root of even multiplicity and
$(t-v)^2$ divides $P_q$.  Comparing the leading, cubic, and constant
coefficients gives
\[
P_q(t)
=(t-v)^2
\left(R_{11} t^2+2(r_1+R_{11} v)t+\frac{R_{22}}{v^2}\right),
\]
\begin{equation}
r_2=v^3 R_{11}+v^2 r_1-\frac{R_{22}}{v},\text{ and }
R_{12}=\frac{R_{22}}{2v^2}-\frac32v^2R_{11}-2vr_1.
\label{eq:corner-direct-coefficients}
\end{equation}
Here recalling that \(r_2\geq0\) in the above description we have
$R_{22}\leq v^4R_{11}+v^3r_1$.

Using the relations \eqref{eq:corner-direct-coefficients}, direct computation yields the following two conic combination descriptions. First, if $R_{22}>v^4R_{11}$, then 
\begin{equation}
\begin{aligned}
q(\vx)={}&R_{11}(x_1-v^2x_2)^2
+2\left(\frac{v^4R_{11} + v^3r_1 - R_{22}}{v^3}\right)h_v(\vx)
+(R_{22}-v^4R_{11}) g^2_{0,1/v}(\vx).
\end{aligned}
\end{equation}
Second, if $R_{22}\leq v^4R_{11}$, then
\begin{equation}
\begin{aligned}
q(\vx)={}&
\left(R_{11}-\frac{R_{22}}{v^4}\right)g^1_{0,v}(\vx)
+\frac{R_{22}}{v^4}(x_1-v^2x_2)^2
+2r_1h_v(\vx).
\end{aligned}
\label{eq:corner-direct-decomposition}
\end{equation}
For both cases, $q$ is a conic combination of an affine square ray, a lifted tangent ray, and a bitangent ray, and the coefficients are all nonnegative. Extremality of $q$ then requires that it is a positive multiple of one of the rays in the combination, excluding the affine squares one since $q$ is non-convex.

\end{proof}

We are now ready to finish the characterization of all extreme rays of $\Ncone(\Fset)$ and prove our major result in Theorem \ref{thm:main}.

\begin{proof}[Proof of Theorem \ref{thm:main}]
    By Propositions \ref{prop:elementary-extreme}, \ref{prop:tangent-extreme}, and \ref{prop:bitangent-extreme}, all rays listed in the theorem are extreme rays of $\Ncone(\Fset)$. It suffices to show that all extreme rays of $\Ncone(\Fset)$ are among such list. Let $q_{R,\vr,\rho}$ be an extreme ray of $\Ncone(\Fset)$. If $R\succeq 0$, then by Proposition \ref{prop:psd-extremes} it is among the list. If $R\not\succeq 0$, then note from the observation \eqref{eq:non-PSD-cond} that either $R_{12}>\sqrt{R_{11}R_{22}}$ or $R_{12}<-\sqrt{R_{11}R_{22}}$. In the case with $R_{12}>\sqrt{R_{11}R_{22}}$, by Proposition \ref{prop:positive-offdiagonal} it is among the list. In the case with $R_{12}<-\sqrt{R_{11}R_{22}}$, by Proposition \ref{prop:no-axis-contact}, Lemma \ref{lem:negative-contact-pattern}, and Propositions \ref{prop:open-axis-hyperbola} and \ref{prop:corner-hyperbola}, it is also among the list.

    It remains to justify that these extreme rays generate the cone. The cone \(\Ncone(\Fset)\) is clearly closed and pointed; we will also show that it has a compact base by defining the linear functional $
        L(q):=\int_{[0,1]^2}q(\vx)\,d\vx.$
    Since \([0,1]^2\subseteq\Fset\), one has \(L(q)>0\) for every
    nonzero \(q\in\Ncone(\Fset)\): otherwise the nonnegative polynomial
    \(q\) would vanish on \([0,1]^2\) and thus vanish identically.  Therefore
    $
        B:=\Set{q\in\Ncone(\Fset)\mid L(q)=1}
    $
    is a compact base of \(\Ncone(\Fset)\).  Indeed, it is closed, and if it
    were unbounded, normalizing an unbounded sequence in \(B\) and taking a
    convergent subsequence would produce a nonzero member of
    \(\Ncone(\Fset)\) on which \(L\) vanishes.  Every point of
    the finite-dimensional compact convex set \(B\) is a convex combination
    of its extreme points, and the extreme points of \(B\) correspond to the
    extreme rays of \(\Ncone(\Fset)\).  Hence \(\Ncone(\Fset)\) is the conic
    hull of the extreme rays listed in the theorem.
\end{proof}

\section{Conclusions}\label{sec:conclusion}

We obtained an exact conic description of the lifted convex hull
$\cC(\Fset)$ from a complete classification of the extreme rays of $\Ncone(\Fset)$. Beyond the rays underlying the standard Shor–RLT
inequalities, the classification identifies the lifted tangent and bitangent
families. Within the Shor–RLT system, the lifted tangent family reduces to
the known second-order-cone constraint $X_{12}^{2}\leq x_1x_2$,
whereas the bitangent families supply the additional relations linking
$X_{12}$ to the diagonal entries $X_{11}$ and $X_{22}$. Although these
inequalities form continuously parameterized families, their combined effect
admits a fixed-size semidefinite representation with only two scalar auxiliary
variables.

The same classification also has algebraic consequences. A nonlinear
parametrization identifies $\Ncone(\Fset)$ with a structured
six-dimensional family of ternary octic forms, and we proved that every
nonnegative form in this family can be written as a sum of squares of ternary
quartic forms. We also derive a bounded-degree preordering certificate for quadratics nonnegative on F and a degree-four
certificate for certain quartic nonnegative on the  half-strip.

The structural foundation of these results is the extreme-ray analysis. The
proof separates the cases of PSD and non-PSD quadratic parts. In the latter
case, boundary reduction and contact analysis show that the bitangent rays
arise precisely from mixed axis–hyperbola contacts. This geometry explains how the bitangent inequalities arise and why they are essential to the exact description of $\cC(\Fset)$.

\section{Acknowledgement}
The first and second authors are partially supported by AFOSR grant FA9550-25-1-0278. This work was supported in part by OpenAI API credits provided by Clemson University and administered by Clemson University Research Computing and Data (RCD). It also used open-weight language models hosted by Clemson University and made available through the Clemson RCD LLM Service.
The OpenAI 5.6 (Sol) model is used in proving Proposition \ref{prop:family-II-sdr}, finding SOS decompositions in Theorem \ref{thm:PisSOS}, the conic combinations in Propositions \ref{prop:open-axis-hyperbola}, \ref{prop:corner-hyperbola}. OpenAI and open-weight language models are also used in polishing the writing in this paper.
\bibliography{article}

\end{document}